\documentclass[final,reqno]{siamltex}
\usepackage{latexsym,amsmath,amssymb,amsfonts,mathrsfs}
\usepackage{graphics,graphicx,subfigure,dsfont,enumerate}
\usepackage{epsf,epsfig,color,cite,cases,multirow,marginnote}
\newcommand{\R}{{\mathbb R}}

\newcommand{\N}{{\mathbb N}}

\newcommand{\Sp}{{\mathbb S}}
\newcommand{\ds}{\displaystyle}
\newcommand{\no}{\nonumber}
\newcommand{\be}{\begin{eqnarray}}
\newcommand{\ben}{\begin{eqnarray*}}
\newcommand{\en}{\end{eqnarray}}
\newcommand{\enn}{\end{eqnarray*}}
\newcommand{\ba}{\backslash}
\newcommand{\pa}{\partial}

\newcommand{\ov}{\overline}

\newcommand{\Rt}{{\rm Re}}
\newcommand{\g}{\gamma}
\newcommand{\G}{\Gamma}

\newcommand{\vep}{\varepsilon}
\newcommand{\Om}{\Omega}
\newcommand{\om}{\omega}

\newcommand{\la}{\lambda}

\newcommand{\ol}{\overline}

\newcommand{\half}{\frac{1}{2}}

\newcommand{\wid}{\widetilde}
\newtheorem{remark}[theorem]{Remark}
\newtheorem{algorithm}{Algorithm}[section]
\makeatletter

\newcommand{\Rmnum}[1]{\expandafter\@slowromancap\romannumeral #1@}
\makeatother

\begin{document}
\renewcommand{\theequation}{\arabic{section}.\arabic{equation}}

%\begin{titlepage}
\title{\bf Uniqueness and direct imaging method for inverse scattering by locally rough surfaces
with phaseless near-field data}
% at a fixed frequency}
%
\author{Xiaoxu Xu\thanks{Academy of Mathematics and Systems Science, Chinese Academy of Sciences,
Beijing 100190, China and School of Mathematical Sciences, University of Chinese
Academy of Sciences, Beijing 100049, China ({\tt xuxiaoxu14@mails.ucas.ac.cn})}
\and
Bo Zhang\thanks{NCMIS, LSEC and Academy of Mathematics and Systems Science, Chinese Academy of
Sciences, Beijing 100190, China and School of Mathematical Sciences, University of Chinese
Academy of Sciences, Beijing 100049, China ({\tt b.zhang@amt.ac.cn})}
\and
Haiwen Zhang\thanks{NCMIS and Academy of Mathematics and Systems Science, Chinese Academy of Sciences,
Beijing 100190, China ({\tt zhanghaiwen@amss.ac.cn})}
}

\date{}

%\end{titlepage}
\maketitle

\begin{abstract}
This paper is concerned with inverse scattering of plane waves by a locally perturbed infinite plane (which is
called a locally rough surface) with the modulus of the total-field data (also called the phaseless near-field
data) at a fixed frequency in two dimensions.
We consider the case where a Dirichlet boundary condition is imposed on the locally
rough surface. This problem models inverse scattering of plane acoustic waves by a one-dimensional sound-soft,
locally rough surface; it also models inverse scattering of plane electromagnetic waves
by a locally perturbed, perfectly reflecting, infinite plane in the TE polarization case.
We prove that the locally rough surface is uniquely determined by the phaseless near-field data generated
by a countably infinite number of plane waves and measured on an open domain above the locally rough surface.
Further, a direct imaging method is proposed to reconstruct the locally rough surface from the phaseless near-field
data generated by plane waves and measured on the upper part of the circle with a sufficiently large radius.
Theoretical analysis of the imaging algorithm is derived by making use of properties of the scattering solution
and results from the theory of oscillatory integrals (especially the method of stationary phase).
Moreover, as a by-product of the theoretical analysis, a similar direct imaging method with full far-field data
is also proposed to reconstruct the locally rough surface.
Finally, numerical experiments are carried out to demonstrate that the imaging algorithm
with phaseless near-field data and full far-field data are fast, accurate and very robust with respect to noise
in the data.



\end{abstract}

\section{Introduction}\label{sec1}
\setcounter{equation}{0}

Acoustic and electromagnetic scattering by a locally perturbed infinite plane (called a locally
rough surface in this paper) occurs in many applications such as radar, remote sensing, geophysics,
medical imaging and nondestructive testing
(see, e.g., \cite{BaoGaoLi2011,BaoLin2011,BurkardPotthast2010,ChenXD18,CZ98,DeSanto1}).

In this paper, we are restricted to the two-dimensional case by assuming that the local perturbation is
invariant in the $x_3$ direction.
Assume further that the incident wave is time-harmonic ($e^{-i\om t}$ time dependence), so that the total
wave field $u$ satisfies the Helmholtz equation
\be\label{eq1}
\Delta u+k^2u=0\quad\mbox{in}\;\;D_+.
\en
Here, $k=\omega/c>0$ is the wave number, $\omega$ and $c$ are the frequency and speed of the wave in $D_+$,
respectively, and $D_+:=\{(x_1,x_2)\;|\;x_2>h(x_1),x_1\in\R\}$ represents a homogeneous medium
above the locally rough surface denoted by $\G:=\pa D_+=\{(x_1,x_2)\;|\;x_2=h(x_1),x_1\in\R\}$
with $h\in C^2(\R)$ having a compact support in $\R$.
In this paper, the incident field $u^i$ is assumed to be the plane wave
\be\label{eq32}
u^i(x,d):=e^{ikx\cdot d},
\en
where $d=(\cos\theta_d,\sin\theta_d)\in \mathbb{S}^1_-$ is the incident direction
with $\pi<\theta_d<2\pi$ and $\Sp^1_-:=\{x=(x_1,x_2)\;|\;|x|=1,x_2<0\}$ is the lower part of the unit
circle $\Sp^1=\{x\in\R^2\;|\;|x|=1\}$.
This paper considers the case where a Dirichlet boundary condition is imposed on the locally rough surface.
Thus, the total field $u(x,d)=u^i(x,d)+u^r(x,d)+u^s(x,d)$ vanishes on the surface $\G$:
\be\label{eq2}
u(x,d)=u^i(x,d)+u^r(x,d)+u^s(x,d)=0\qquad\mbox{on}\;\;\G,
\en
where $u^r$ is the reflected wave by the infinite plane $x_2=0$:
\be\label{eq33}
u^r(x,d):=-e^{ikx\cdot d'}
\en
with $d^\prime=(\cos\theta_d,-\sin\theta_d)$ and $u^s$ is the unknown scattered wave to be determined
which is required to satisfy the Sommerfeld radiation condition
\be\label{rc}
\lim_{r\to\infty}r^{\frac12}\left(\frac{\pa u^s}{\pa r}-iku^s\right)=0,\quad r=|x|,\quad x\in D_+.
\en
This problem models electromagnetic scattering by a locally perturbed, perfectly conducting, infinite plane
in the TE polarization case; it also models acoustic scattering by a one-dimensional sound-soft,
locally rough surface. See FIG. \ref{fig6} for the geometry of the scattering problem.

\begin{figure}
\centering
\includegraphics[width=3in]{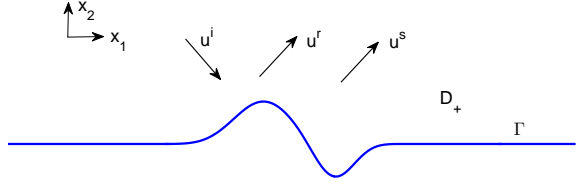}
  \vspace{-0.2in}
\caption{The scattering problem from a locally rough surface}\label{fig6}
\end{figure}

The well-posedness of the scattering problem (\ref{eq1})-(\ref{rc}) has been studied by using
the variational method with a Dirichlet-to-Neumann (DtN) map in \cite{BaoLin2011} or the integral
equation method in \cite{Willers1987,ZZ13}. In particular, it was proved in \cite{Willers1987,ZZ13} that
$u^s$ has the following asymptotic behavior at infinity:
\be\label{eq96}
u^s(x,d)=\frac{e^{ik|x|}}{\sqrt{|x|}}\left(u^\infty(\hat{x},d)+O\Big(\frac{1}{|x|}\Big)\right),
\qquad |x|\to\infty
\en
uniformly for all observation directions $\hat{x}:=x/|x|\in\Sp^1_+$ with
$\Sp^1_+:=\{x=(x_1,x_2)\;|\;|x|=1,x_2>0\}$ the upper part of the unit circle $\Sp^1$,
where $u^\infty(\hat{x},d)$ is called the far-field pattern of the scattered field $u^s$,
depending on the observation direction $\hat{x}\in\Sp^1_+$ and the incident direction $d\in\Sp^1_-$.

Many numerical algorithms have been proposed for the inverse problem of reconstructing
the rough surfaces from the phased near-field or far-field data (see, e.g.,
\cite{BaoLin2011,BurkardPotthast2010,CL05,CGHIR,DeSanto1,DeSanto2,DLLY,KressTran2000,LSZ17,LZZ18,Spivack,ZZ13}
and the references quoted there). For the case when the local perturbation is below
the infinite plane which is called the inverse cavity problem, see \cite{BaoGaoLi2011,Li2012} and
the reference quoted there.

In diffractive optics and radar imaging, it is much harder to obtain data with accurate phase information
compared with only measuring the intensity (or the modulus) of
the data \cite{BLL13,BLT15,Candes15,ChenXD18,CH16,FDAM06,KR97,Pan11}.
Thus it is often desirable to study inverse scattering problems with phaseless data.
Inverse scattering with phaseless near-field data has been extensively studied numerically over the past
decades (see, e.g., \cite{BLL13,BZ16,Candes15,CH2017,CH16,CFH17,FDAM06,NMP,Pan11,ChenXD17}
and the references quoted there). Recently, mathematical issues including uniqueness and stability
have also been studied for inverse scattering with phaseless near-field data
(see, e.g., \cite{K14,K17,KR16,KR17,MH17,N15,N16} and the references quoted there).

In contrast to the case with phaseless near-field data, inverse scattering with phaseless far-field data
is much less studied both mathematically and numerically due to the {\em translation invariance property}
of the phaseless far-field data, that is, the modulus of the far-field pattern is invariant under
translations of the obstacle for plane wave incidence \cite{KR97,LS04,ZZ17}.
The translation invariance property makes it impossible to reconstruct the location of the obstacle or
the inhomogeneous medium from the phaseless far-field pattern with one plane wave as the incident field.
Nevertheless, several reconstruction algorithms have been developed to
reconstruct the shape of the obstacle from the phaseless far-field data with one plane
wave as the incident field (see \cite{ACZ16,I07,IK10,IK11,KR97,LL15,LLW17,S16}).
Uniqueness has also been established in recovering the shape of the obstacle from the phaseless far-field
data with one plane wave as the incident field \cite{LZ10,majda76}.
Recently, progress has been made on the mathematical and numerical study of inverse scattering with
phaseless far-field data. For example, it was first proved in \cite{ZZ17} that the translation invariance
property of the phaseless far-field pattern can be broken by using superpositions of two plane waves as
the incident fields for all wave numbers in a finite interval. And a recursive Newton-type iteration algorithm
in frequencies was further developed in \cite{ZZ17} to numerically reconstruct both the location and the
shape of the obstacle simultaneously from multi-frequency phaseless far-field data.
This method was further extended in \cite{ZZ17b} to reconstruct the locally rough surface from
multi-frequency intensity-only far-field or near-field data.
Furthermore, a direct imaging algorithm was recently developed in \cite{ZZ18} to reconstruct the obstacle
from the phaseless far-field data generated by infinitely many sets of superpositions of two plane waves as
the incident fields at a fixed frequency. And uniqueness results have also been established rigorously in
\cite{XZZ18} for inverse obstacle and medium scattering from the phaseless far-field patterns generated
by infinitely many sets of superpositions of two plane waves with different directions at a fixed frequency
under certain a priori conditions on the obstacle and the inhomogeneous medium.
The a priori assumption on the obstacle and the inhomogeneous medium in \cite{XZZ18} was removed
in \cite{XZZ18b} by adding a known reference ball into the scattering model.
Note that the idea of adding a reference ball to the scattering system was recently used in \cite{ZG18}
to prove uniqueness results for inverse scattering with phaseless far-field data generated by
superpositions of a plane wave and a point source as the incident fields at a fixed frequency.
Note further that, by adding one point scatterer into the scattering model stability estimates have been
obtained in \cite{JLZ18} for inverse obstacle and medium scattering with phaseless far-field data
associated with one plane wave as the incident field under certain conditions on the obstacle and
inhomogeneous medium if the point scatterer is placed far away from the scatterer.
In addition, direct imaging algorithms are proposed in \cite{JLZ18} to reconstruct the scattering obstacle
from the phaseless far-field data associated with one plane wave as the incident field.

In this paper, we consider uniqueness and fast imaging algorithm for inverse scattering by locally rough
surfaces from phaseless near-field data corresponding to incident plane waves at a fixed frequency.
First, we prove that the locally rough surface is uniquely determined by the phaseless near-field data
generated by a countably infinite number of incident plane waves and measured on an open domain above
the locally rough surface, following the ideas in \cite{ZZ13,N15}.
Then we develop a direct imaging algorithm for the inverse scattering problem with phaseless near-field data
generated by incident plane waves and measured on the upper part of the circle containing the local
perturbation part of the infinite plane, based on the imaging function $I^{Phaseless}(z)$ with $z\in\R^2$
(see the formula (\ref{eq3}) below). The theoretical analysis of the imaging function $I^{Phaseless}(z)$
is given by making use of properties of the scattering solution and results from the theory of
oscillatory integrals (especially the method of stationary phase).
From the theoretical analysis result, it is expected that if the radius of the measurement circle is
sufficiently large, $I^{Phaseless}(z)$ will take a large value when $z$ is on the boundary $\G$ and
decay as $z$ moves away from $\G$. Based on this, a direct imaging algorithm is proposed to recover
the locally rough surface from the phaseless near-field data.
Further, numerical experiments are also carried out to demonstrate that our imaging algorithm
provides an accurate, fast and stable reconstruction of the locally rough surface.
Moreover, as a by-product of the theoretical analysis, a similar direct imaging algorithm with
full far-field data is also proposed to reconstruct the locally rough surfaces with convincing numerical
experiments illustrating the effectiveness of the imaging algorithm.
It should be pointed out that a direct imaging method was recently proposed in \cite{CH2017,CH16}
for reconstructing extended obstacles with acoustic and electromagnetic phaseless near-field data,
based on the reverse time migration technique

The remaining part of the paper is organized as follows. The uniqueness result is proved in
Section \ref{sec2} for an inverse scattering problem with phaseless near-field data.
In Section \ref{sec3}, the direct imaging method with phaseless near-field data is proposed, and
its theoretical analysis is given. As a by-product, the direct imaging method with full far-field data
is also presented in Section \ref{sec3}.
Numerical experiments are carried out in Section \ref{sec4} to illustrate the effectiveness
of the imaging method. Conclusions are given in Section \ref{conclusion}.
In Appendix A, we use the method of stationary phase to prove Lemma \ref{lem5} in Section \ref{sec3}
which plays an important role in the theoretical analysis of the direct imaging method.

We conclude this section with introducing some notations used throughout this paper.
Define $B_R:=\{x=(x_1,x_2)\;|\;|x|<R\}$ to be a disk centered at the origin and with radius $R>0$
large enough so that the local perturbation $\G_p:=\{(x_1,h(x_1))\;|\;x_1\in\textrm{supp}(h)\}\subset B_R$.
Define $\R^2_\pm:=\{(x_1,x_2)\in\R^2\;|\;x_2\gtrless0\}$, $\pa B^+_R:=\pa B_R\cap D_+$.
For any $x,z\in\R^2$ and $d\in\Sp^1$, set $x:=(x_1,x_2),z:=(z_1,z_2),d:=(d_1,d_2)$
and let $x':= (x_1,-x_2)$ be the reflection of $x$ with respect to the $x_1$-axis.
Further, let $\hat{x}=x/|x|=(\hat{x}_1,\hat{x}_2)=(\cos\theta_{\hat{x}},\sin\theta_{\hat{x}})$,
$\hat{z}=z/|z|=(\hat{z}_1,\hat{z}_2)=(\cos\theta_{\hat{z}},\sin\theta_{\hat{z}})$
and $d=(\cos\theta_d,\sin\theta_d)$ with $\theta_{\hat{x}},\theta_{\hat{z}},\theta_d\in[0,2\pi]$.
Note also that if $x\neq0$ then $\hat{x}_1=x_1/|x|$ and $\hat{x}_2=x_2/|x|$.
Throughout this paper, the positive constants $C$, $C_1$ and $C_2$ may be different at different places.

\section{Uniqueness for an inverse problem}\label{sec2}
\setcounter{equation}{0}

In this section, we establish a uniqueness result for an inverse scattering problem with phaseless
near-field data, motivated by \cite{N15}.
To this end, assume that $\G_1,\G_2$ are two locally rough surfaces, where
$\G_j:=\{(x_1,x_2)\;|\;x_2=h_j(x_1),x_1\in\R\}$ with $h_j\in C^2(\R)$ having a compact support in $\R$, $j=1,2$.
Further, denote by $\G_{p,j}:=\{(x_1,h_j(x_1))\;|\;x_1\in\textrm{supp}(h_j)\}$ the local perturbation of $\G_j$
and by $D_{+,j}$ the domain above $\G_j$, $j=1,2$.
For $j=1,2$ suppose that the total field is given by $u_j=u^r+u^r+u^s_j$, where $u^s_j(\hat{x},d)$ is the
scattered field corresponding to the locally rough surface $\G_j$ with its far-field pattern $u^\infty_j(\hat{x},d)$.
Moreover, let $R>0$ be large enough such that the local perturbation $\G_{p,j}\subset B_R$ ($j=1,2$)
and let $\Omega$ be a bounded open domain above the locally rough surfaces $\G_1$ and $\G_2$.
See FIG. \ref{fig7} for the geometry of the inverse scattering problem.
\begin{figure}
\centering
\includegraphics[width=3in]{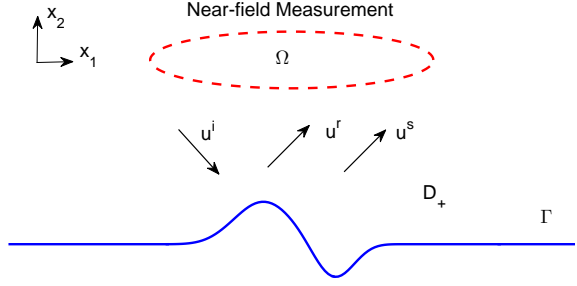}
 \vspace{-0.2in}
\caption{Inverse scattering with phaseless near-field data measured on the domain $\Omega$
}\label{fig7}
\end{figure}

We need the following result on the property of the scattered field which
is also useful in the numerical algorithm in Section \ref{sec3}.

\begin{lemma}\label{le3}
Let $x\in D_+,\,d\in\Sp^1_-$. Then for any $x\in D_+$ with $|x|$ large enough and $d\in\Sp^1_-$
the scattering solution $u^s(x,d)$ of the scattering problem (\ref{eq1})-(\ref{rc}) has the asymptotic
behavior
\be\label{eq92}
u^s(x,d)=\frac{e^{ik|x|}}{|x|^{1/2}}u^\infty(\hat{x},d)+u^s_{Res}(x,d)
\en
with
\be\label{eq94-}
\|u^\infty(\cdot,d)\|_{C^1({\Sp^1_+})}&\le& C,\\ \label{eq94}
|u^s_{Res}(x,d)|&\le& \frac{C}{|x|^{3/2}},
\en
where $C>0$ is a constant independent of $x$ and $d$.
\end{lemma}

\begin{proof}
The statement of this lemma follows easily from the well-posedness of the scattering
problem (\ref{eq1})-(\ref{rc}) and the asymptotic behavior (\ref{eq96}) of the scattered field $u^s$
(see, e.g., \cite{ZZ13}).
\end{proof}

We also need the following uniqueness result for the inverse scattering problem with full far-field data
which is given in \cite{ZZ13}.
% (see also \cite{Ma05} for a similar result in inverse cavity problems).

\begin{theorem}[Theorem 4.1 in \cite{ZZ13}]\label{thm-uni-full}
Assume that $\G_1$ and $\G_2$ are two locally rough surfaces and
$u^\infty_1(\hat{x},d)$ and $u^\infty_2(\hat{x},d)$ are the far-field patterns
corresponding to $\G_1$ and $\G_2$, respectively.
If $u^\infty_1(\hat{x},d_n)=u^\infty_2(\hat{x},d_n)$ for all $\hat{x}\in\Sp^1_+$
and the distinct directions $d_n\in\Sp^1_-$ with $n\in\N$ and a fixed wave number $k$,
then $\G_1=\G_2$.
\end{theorem}

We are now ready to state and prove the main theorem of this section.

\begin{theorem}\label{thm-uni-phaseless}
Assume that $\G_1$ and $\G_2$ are two locally rough surfaces and
$u_1(\hat{x},d)$ and $u_2(\hat{x},d)$ are the total field corresponding to $\G_1$ and $\G_2$, respectively.
Let $\Omega$ be a bounded open domain above $\G_1$ and $\G_2$.
If $|u_1(x,d_n)|=|u_2(x,d_n)|$ for all $x\in\Omega$
and the distinct directions $d_n\in\Sp^1_-$ with $n\in\N$ and a fixed wave number $k$,
then $\G_1=\G_2$.
\end{theorem}

\begin{proof}
Fix $d=d_n$ for an arbitrary $n\in\N$ and set $d=(d_1,d_2)$. Since $|u_1(x,d)|=|u_2(x,d)|$ for all $x\in\Om$,
it follows from the analyticity of $|u_l(x,d)|^2$, $l=1,2$, with respect to $x\in\R^2_+\ba\ov{B_R}$ that
\be\label{eq34}
|u_1(x,d)|=|u_2(x,d)|\quad\textrm{for}\;\;x\in\R^2_+\ba\ov{B_R}.
\en
Noting that $u_l=u^i+u^r+u^s_l$, $l=1,2$, we have
\be\label{eq31}
|u_l|^2=|u^i+u^r+u^s_l|^2=|u^i+u^r|^2+|u^s_l|^2+2\Rt(u^i\ov{u^s_l})+2\Rt(u^r\ov{u^s_l}).
\en
Now, by Lemma \ref{le3} we know that for $x\in D_{+,l}$,
\be\label{eq28}
u^s_l(x,d)=\frac{e^{ik|x|}}{|x|^{1/2}}u^\infty_l(\hat{x},d)+u^s_{l,Res}(x,d),\;\;l=1,2
\en
with
\be\label{eq29}
|u^s_{l,Res}(x,d)|\le{C}{|x|^{-3/2}},\qquad
|u^s_l(x,d)|\le{C}{|x|^{-1/2}}
\en
for $|x|$ large enough.

Write
\be\label{eq30}
u^\infty_l(\hat{x},d)=r_l(\hat{x},d)e^{i\theta_l(\hat{x},d)},\quad l=1,2,
\en
where $r_l(\hat{x},d),\theta_l(\hat{x},d)$ are real-valued functions with $r_l\ge0$ and $\theta_l\in[0,2\pi]$.
Then, by inserting (\ref{eq28}) and (\ref{eq30}) into (\ref{eq31}) we obtain that for $l=1,2$,
\ben
|u_l(x,d)|^2&=&|u^i(x,d)+u^r(x,d)|^2+|u^s_l(x,d)|^2+2\Rt\left(u^i(x,d)\ov{u^s_{l,Res}(x,d)}\right)\\
&&+2\Rt\left(u^i(x,d)\frac{e^{-ik|x|}}{|x|^{1/2}}r_l(\hat{x},d)e^{-i\theta_l(\hat{x},d)}\right)
+2\Rt\left(u^r(x,d)\ov{u^s_{l,Res}(x,d)}\right)\\
&&+2\Rt\left(u^r(x,d)\frac{e^{-ik|x|}}{|x|^{1/2}}r_l(\hat{x},d) e^{-i\theta_l(\hat{x},d)}\right).
\enn
This yields
\be\label{eq97}
&&\frac{|x|^{1/2}}{2}\left(|u_l(x,d)|^2-|u^i(x,d)+u^r(x,d)|^2\right)\no\\
&&\qquad\quad =\Rt\left(u^i(x,d) r_l(\hat{x},d) e^{-i(k|x|+\theta_l(\hat{x},d))}\right)\no\\
&&\qquad\qquad\;\;+\Rt\left(u^r(x,d)r_l(\hat{x},d)e^{-i(k|x|+\theta_l(\hat{x},d))}\right)+v_l(x,d),
\en
where $v_l$ is given by
\ben
v_l(x,d)=|x|^{1/2}\left[\frac{1}{2}|u^s_l(x,d)|^2+\Rt\left(u^i(x,d)\ov{u^s_{l,Res}(x,d)}\right)
+\Rt\left(u^r(x,d)\ov{u^s_{l,Res}(x,d)}\right)\right].
\enn
Further, by (\ref{eq29}) we see that for $l=1,2$,
\be\label{eq38}
|v_l(x,d)|\leq\frac{C}{|x|^{1/2}}\quad\textrm{as}\;\;|x|\rightarrow+\infty.
\en
Substituting (\ref{eq32}) and (\ref{eq33}) into (\ref{eq97}) gives that for $x\in\R^2_+\ba\ov{B_R}$,
\be\label{eq35}
&&\frac{|x|^{1/2}}{4}(|u_l(x,d)|^2-|u^i(x,d)+u^r(x,d)|^2)\no\\
&&\quad=\frac{1}{2}r_l(\hat{x},d)\left[\cos(kx\cdot d-k|x|-\theta_l(\hat{x},d))
-\cos(k x\cdot d'-k|x|-\theta_l(\hat{x},d))\right]+\frac{1}{2}v_l(x,d)\no\\
&&\quad=r_l(\hat{x},d)\sin(k\hat{x}_2d_2|x|)\sin(\theta_l(\hat{x},d)
+|x|(k-k\hat{x}_1 d_1))+\frac{1}{2}v_l(x,d),\;\;l=1,2.
\en
Thus, and by (\ref{eq34}) we have that for $x\in\R^2_+\ba\ov{B_R}$,
\be\label{eq36}
&&r_1(\hat{x},d)\sin(k\hat{x}_2d_2|x|)\sin[\theta_1(\hat{x},d)+|x|(k-k\hat{x}_1d_1)]+\frac{1}{2}v_1(x,d)\no\\
&&\qquad=r_2(\hat{x},d)\sin(k\hat{x}_2d_2|x|)\sin[\theta_2(\hat{x},d)+|x|(k-k\hat{x}_1d_1)]+\frac{1}{2}v_2(x,d).
\en

Arbitrarily fix $\hat{x}=(\hat{x}_1,\hat{x}_2)\in\Sp^1_+$ and set $\alpha=k\hat{x}_2d_2$ and
$\beta=k(1-\hat{x}_1d_1)$. The equation (\ref{eq36}) then becomes
\be\no
&&r_1(\hat{x},d)\sin(\alpha|x|)\sin[\theta_1(\hat{x},d)+\beta|x|]+\frac{1}{2}v_1(x,d)\\ \label{eq37}
&&\qquad=r_2(\hat{x},d)\sin(\alpha|x|)\sin[\theta_2(\hat{x},d)+\beta|x|]+\frac{1}{2}v_2(x,d).
\en
Note that $\alpha<0,\beta>0$ since $\hat{x}=(\hat{x}_1,\hat{x}_2)\in\Sp^1_+$ and $d=(d_1,d_2)\in\Sp^1_-$.
Then we can choose $\g^{(1)}_0,\g^{(2)}_0\in\R$ such that
\be\label{eq39}
\sin\left(\frac{\alpha}{\beta}\g^{(k)}_0\right)\neq0,\;\;k=1,2,\\ \label{eq41}
\sin(\g^{(1)}_0-\g^{(2)}_0)\neq0.
\en
We now prove that
\be\label{eq40}
r_1\sin(\theta_1+\g^{(k)}_0)=r_2\sin(\theta_2+\g^{(k)}_0),\;\;k=1,2,
\en
where we write $r_l=r_l(\hat{x},d),$ $\theta_l=\theta_l(\hat{x},d)$, $l=1,2$, for simplicity.
We distinguish between the following two cases.

{\bf Case 1.} $\alpha/\beta$ is a rational number. In this case, it is easily seen that there exist $p_j\in\N$
with $j=1,2,\ldots$ such that $({\alpha}/{\beta})p_j\in\N$ and $\lim\limits_{j\rightarrow+\infty}p_j=+\infty$.
For $k=1,2$ let $x^{(k)}_j:=(\g^{(k)}_0+2\pi p_j)\hat{x}/\beta$.
Then it is easy to see that $x^{(k)}_j\in\R^2_+\ba\ov{B}_R$ for large $j$
and $\lim\limits_{j\rightarrow+\infty}|x^{(k)}_j|=+\infty$.
Thus, take $x=x^{(k)}_j$ with large $j$ in (\ref{eq37}) to obtain that
\ben
&&r_1\sin\left(\frac{\alpha}{\beta}\g^{(k)}_0\right)\sin(\theta_1+\g^{(k)}_0)+\frac{1}{2}v_1(x^{(k)}_j,d)\\
&&\qquad\quad=r_2\sin\left(\frac{\alpha}{\beta}\g^{(k)}_0\right)\sin(\theta_2+\g^{(k)}_0)
+\frac{1}{2}v_2(x^{(k)}_j,d).
\enn
The required equality (\ref{eq40}) then follows by taking $j\rightarrow+\infty$ in the above equation
and using (\ref{eq38}) and (\ref{eq39}).

{\bf Case 2.} $\alpha/\beta$ is an irrational number. In this case, by Kronecker's approximation theorem
(see, e.g., \cite[Theorem 7.7]{A90}), we know that there exist $p_j\in\N$ with $j=1,2,\ldots$ such that
$({\alpha}/{\beta})p_j=m_j+a_j$ with $m_j\in\N$, $\lim\limits_{j\rightarrow+\infty}a_j=0$
and $\lim\limits_{j\rightarrow+\infty}p_j=+\infty$. For $k=1,2$ let $x^{(k)}_j$ be defined as in Case 1.
Then, similarly as in Case 1, take $x=x^{(k)}_j$ with large $j$ in (\ref{eq37}) to deduce that
\ben
&&r_1\sin(\frac{\alpha}{\beta}\g^{(k)}_0+2\pi a_j)\sin(\theta_1+\g^{(k)}_0)+\frac{1}{2}v_1(x^{(k)}_j,d)\\
&&\qquad\quad=r_2\sin(\frac{\alpha}{\beta}\g^{(k)}_0+2\pi a_j)\sin(\theta_2+\g^{(k)}_0)
+\frac{1}{2}v_2(x^{(k)}_j,d).
\enn
Thus, (\ref{eq40}) also follows by letting $j\rightarrow+\infty$ in the above equation
and using (\ref{eq38}) and (\ref{eq39}).

Finally, it follows from (\ref{eq40}) and the arbitrariness of $\hat{x},d$ that
{\small
\ben
\left(\begin{array}{ll}
\cos\gamma^{(1)}_0(\hat{x},d_n)&\sin\gamma^{(1)}_0(\hat{x},d_n)\\
\cos\gamma^{(2)}_0(\hat{x},d_n)&\sin\gamma^{(2)}_0(\hat{x},d_n)
\end{array}\right)
\left(\begin{array}{l}
r_1(\hat{x},d_n)\sin\theta_1(\hat{x},d_n)-r_2(\hat{x},d_n)\sin\theta_2(\hat{x},d_n)\\
r_1(\hat{x},d_n)\cos\theta_1(\hat{x},d_n)-r_2(\hat{x},d_n)\cos\theta_2(\hat{x},d_n)
\end{array}\right)=0
\enn
}
for all $\hat{x}\in\Sp^1_+$ and $d_n\in\Sp^1_-$ with $n\in\N$.
Condition (\ref{eq41}) means that the determinant of the square matrix on the left of the above matrix equation
does not vanish, and so the above matrix equation only has a trivial solution, that is,
\ben
r_1(\hat{x},d_n)\sin\theta_1(\hat{x},d_n)&=&r_2(\hat{x},d_n)\sin\theta_2(\hat{x},d_n),\\
r_1(\hat{x},d_n)\cos\theta_1(\hat{x},d_n)&=&r_2(\hat{x},d_n)\cos\theta_2(\hat{x},d_n)
\enn
for all $\hat{x}\in\Sp^1_+$ and $d_n\in\Sp^1_-$ with $n\in\N$.
This implies that $u^\infty_1(\hat{x},d_n)=u^\infty_2(\hat{x},d_n)$ for all $\hat{x}\in\Sp^1_+$
and $d_n\in\Sp^1_-$ with $n\in\N$. The required result then follows from Theorem \ref{thm-uni-full}.
The proof is thus completed.
\end{proof}

\section{Direct imaging method for inverse problems}\label{sec3}
\setcounter{equation}{0}

In this section, we consider the inverse problem: Given the incident field $u^i=u^i(x,d)$,
to reconstruct the locally rough surface $\G$ from the phaseless near-field data $|u(x,d)|$
for all $x\in\pa B^+_R, d\in\Sp^1_-$ and with a fixed wave number $k$.
See FIG. \ref{fig8} for the geometry of the inverse scattering problem.
Our purpose is to develop a direct imaging method to solve this inverse problem numerically
though no rigorous uniqueness result is available yet for the inverse problem.
\begin{figure}
\centering
\includegraphics[width=3in]{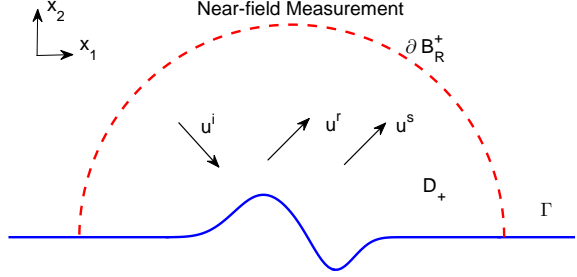}
 \vspace{-0.2in}
\caption{Inverse scattering with phaseless near-field data measured on the curve $\pa B^+_R$
}\label{fig8}
\end{figure}

We consider the imaging function
\be\label{eq3}
&&I^{Phaseless}(z)\no\\
&&\quad:=\int_{\pa B^+_R}\left|
\int_{\Sp^1_-}\left[\left(|u(x,d)|^2-2+e^{2ikx_2d_2}\right)e^{ik(x-z)\cdot d}
-e^{ik(x'-z')\cdot d}\right]ds(d)\right|^2dx\;\quad
\en
for $z\in\R^2$. In what follows, we will study the behavior of this imaging function.

Define
\be\label{eq16}
U(x,z)&:=&U_1(x,z)+U_2(x,z)+U_3(x,z),\\ \label{eq16+}
W(x,z)&:=&W_1(x,z)+W_2(x,z)+W_3(x,z)+W_4(x,z),
\en
where
\be\label{eq44}
U_1(x,z)&=&\int_{\Sp^1_-} u^s(x,d)e^{-ik z\cdot d} ds(d),\\ \label{eq45}
U_2(x,z)&=&-\int_{\Sp^1_-}e^{ik(x\cdot d'-z\cdot d)}ds(d),\\ \label{eq46}
U_3(x,z)&=&-\int_{\Sp^1_-}e^{ik(x\cdot d'-z'\cdot d)}ds(d),
\en
and
\ben
W_1(x,z)&=&\int_{\Sp^1_-}\left[u^i(x,d)\right]^2\ov{u^s(x,d)}e^{-ikz\cdot d}ds(d),\\
W_2(x,z)&=&\int_{\Sp^1_-}u^i(x,d)u^r(x,d)\overline{u^s(x,d)}e^{-ikz\cdot d}ds(d),\\
W_3(x,z)&=&\int_{\Sp^1_-}u^i(x,d)\overline{u^r(x,d)}u^s(x,d)e^{-ikz\cdot d}ds(d),\\
W_4(x,z)&=&\int_{\Sp^1_-}u^i(x,d)|u^s(x,d)|^2e^{-ikz\cdot d}ds(d).
\enn
Since $u=u^i+u^r+u^s$ and $|u^i|=|u^r|=1$, by a direct calculation (\ref{eq3}) becomes
\be\label{eq89}
I^{Phaseless}(z)=\int_{\pa B^+_R}|U(x,z)+W(x,z)|^2dx.
\en

We need the following result for oscillatory integrals proved in \cite{CH2017}.

\begin{lemma}[Lemma 3.9 in \cite{CH2017}]\label{le1}
For any $-\infty<a<b<\infty$ let $u\in C^2[a,b]$ be real-valued and satisfy that
$|u'(t)|\geq1$ for all $t\in(a,b)$. Assume that $a=x_0<x_1<\cdots<x_N=b$ is a division of $(a,b)$
such that $u'$ is monotone in each interval $(x_{i-1},x_i)$, $i=1,\ldots,N$.
Then for any function $\phi$ defined on $(a,b)$ with integrable derivative and for any $\la>0$,
\ben
\left|\int_a^b e^{i\la u(t)}\phi(t)dt\right|\leq(2N+2)\la^{-1}\left[|\phi(b)|+\int^b_a|\phi'(t)|dt\right].
\enn
\end{lemma}

With the aid of Lemma \ref{le1}, we can obtain the following lemma.

\begin{lemma}\label{le2}
Let $x\in\R^2_+,\,d\in\Sp^1_-$. For $\hat{x}=x/|x|\in\Sp^1_+$ assume that
$f(\hat{x},\cdot),$ $g(\hat{x},\cdot)\in C^1(\ov{\Sp^1_-})$ and define
\ben\label{eq4}
F(x):=\int_{\Sp^1_-}e^{ikx\cdot d}f(\hat{x},d)ds(d),\qquad
G(x):=\int_{\Sp^1_-}e^{ikx\cdot d'}g(\hat{x},d)ds(d).
\enn
Then for all $x\in{\R^2_+}$ with $|x|$ large enough we have
\be\label{eq80}
|F(x)|\leq C{\|f(\hat{x},\cdot)\|_{C^1(\ov{\Sp^1_-})}}{|x|^{-1/2}},\\ \label{eq81}
|G(x)|\leq C{\|g(\hat{x},\cdot)\|_{C^1(\ov{\Sp^1_-})}}{|x|^{-1/2}},
\en
where $C>0$ is a constant independent of $x$.
\end{lemma}

\begin{proof}
We only prove (\ref{eq80}). The proof of (\ref{eq81}) is similar.

Let $\delta>0$ be small enough such that $\sin\delta\ge\delta/2$ and let $|x|$ be large enough.
Let $x=|x|\hat{x}=|x|(\cos\theta_{\hat{x}},\sin\theta_{\hat{x}})$,
$d=(\cos\theta_d,\sin\theta_d)$ with $\theta_{\hat{x}}\in[0,\pi]$, $\theta_d\in[\pi,2\pi]$, and
define $\wid{f}(\theta_{\hat{x}},\theta_d):=f(\hat{x},d)$ for $\theta_{\hat{x}}\in [0,\pi]$ and
$\theta_d\in[\pi,2\pi]$. Then it follows that
\be\label{eq7}
C_1\|f(\hat{x},\cdot)\|_{C^1(\ov{\Sp^1_-})}\leq\|\wid{f}(\theta_{\hat{x}},\cdot)\|_{C^1[\pi,2\pi]}
\le C_2\|f(\hat{x},\cdot)\|_{C^1(\ov{\Sp^1_-})}
\en
and
\be\label{eq8}
F(x)=\int^{2\pi}_{\pi}e^{ik|x|\cos(\theta_d-\theta_{\hat{x}})}\tilde{f}(\theta_{\hat{x}},\theta_d)d\theta_d.
\en
We distinguish between the following two cases.

{\bf Case 1.} $\theta_{\hat{x}}\in[0,\delta]\cup[\pi-\delta,\pi]$. In this case, we rewrite (\ref{eq8}) as
\be\label{eq98}
F(x)&=&\int^{2\pi-2\delta}_{\pi+2\delta}e^{ik|x|\cos(\theta_d-\theta_{\hat{x}})}
\wid{f}(\theta_{\hat{x}},\theta_d)d\theta_d
+\int^{\pi+2\delta}_{\pi}e^{ik|x|\cos(\theta_d-\theta_{\hat{x}})}
\wid{f}(\theta_{\hat{x}},\theta_d)d\theta_d\no\\
&& +\int^{2\pi}_{2\pi-\delta}e^{ik|x|\cos(\theta_d-\theta_{\hat{x}})}
  \wid{f}(\theta_{\hat{x}},\theta_d)d\theta_d\no\\
&:=&I_1+II_1+III_1.
\en
Set $u(\theta_d)=2\cos(\theta_d-\theta_{\hat{x}})/\delta$.
Then $u'(\theta_d)=-2\sin(\theta_d-\theta_{\hat{x}})/\delta$, and so,
for $\theta_{\hat{x}}\in[0,\delta]\cup[\pi-\delta,\pi]$ and $\theta_d\in[\pi+2\delta,2\pi-2\delta]$ we have
\ben
|u'(\theta_d)|=2|\sin(\theta_d-\theta_{\hat{x}})|/\delta\geq 2|\sin\delta|/\delta
\ge\frac{2}{\delta}\frac{\delta}{2}=1
\enn
and $u'(\theta_d)$ is monotone in $[\pi+2\delta,2\pi-2\delta]$.
Thus, by Lemma \ref{le1} it follows that
\be\label{eq5}
|I_1|=\left|\int^{2\pi-2\delta}_{\pi+2\delta}e^{i\frac{\delta k|x|}{2}u(\theta_d)}
\wid{f}(\theta_{\hat{x}},\theta_d)d\theta_d\right|
\le C\frac{\|\wid{f}(\theta_{\hat{x}},\cdot)\|_{C^1[\pi,2\pi]}}{\delta|x|}.
\en
It is easy to obtain that
\be\label{eq6}
|II_1|+|III_1|\leq {C}{\delta}\|\wid{f}(\theta_{\hat{x}},\cdot)\|_{C^1[\pi,2\pi]}.
\en
Combining (\ref{eq7}), (\ref{eq98}), (\ref{eq5}) and (\ref{eq6}) gives
\ben
|F(x)|\leq C\left(\frac{1}{\delta|x|}+\delta\right)
\|\wid{f}(\theta_{\hat{x}},\cdot)\|_{C^1[\pi,2\pi]}
\leq C\left(\frac{1}{\delta|x|}+\delta\right)\|f(\hat{x},\cdot)\|_{C^1(\ov{\Sp^1_-})}.
\enn
From this (\ref{eq80}) follows immediately on taking $\delta=|x|^{-1/2}$.

{\bf Case 2.} $\theta_{\hat{x}}\in[\delta,\pi-\delta]$. In this case, we rewrite (\ref{eq8}) as
\ben
F(x)&=&\int^{\theta_{\hat{x}}-\delta}_{0}e^{ik|x|\cos(\theta_d-\theta_{\hat{x}})}
\wid{f}(\theta_{\hat{x}},\theta_d)d\theta_d
+\int^{\pi}_{\theta_{\hat{x}}+\delta}e^{ik|x|\cos(\theta_d-\theta_{\hat{x}})}
\wid{f}(\theta_{\hat{x}},\theta_d)d\theta_d\\
&&+\int^{\theta_{\hat{x}}+\delta}_{\theta_{\hat{x}}-\delta}e^{ik|x|\cos(\theta_d-\theta_{\hat{x}})}
\wid{f}(\theta_{\hat{x}},\theta_d)d\theta_d\\
&:=&I_2+II_2+III_2.
\enn
Similarly as in the estimation of $I_1$, it is deduced that
\ben
|I_2|+|II_2|\leq \frac{C}{\delta|x|}\|\wid{f}(\theta_{\hat{x}},\cdot)\|_{C^1[\pi,2\pi]}.
\enn
Now, it is straightforward to see that
\ben
|III_2|\le C\delta\|\wid{f}(\theta_{\hat{x}},\cdot)\|_{C^1[\pi,2\pi]}.
\enn
Then we arrive at
\ben
|F(x)|\leq C\left(\frac{1}{\delta|x|}+\delta\right)\|\wid{f}(\theta_{\hat{x}},\cdot)\|_{C^1[\pi,2\pi]}
\leq C\left(\frac{1}{\delta|x|}+\delta\right)\|f(\hat{x},\cdot)\|_{C^1(\ov{\Sp^1_-})}.
\enn
Taking $\delta=|x|^{-1/2}$ in the above inequality gives (\ref{eq80}).
The proof is thus completed.
\end{proof}

We also need the following reciprocity relation of the far-field pattern.

\begin{lemma}\label{lem13}
For $\hat{x}\in{\Sp^1_+},\;d\in{\Sp^1_-}$ let $u^\infty(\hat{x},d)$ denote the far-field pattern
of the scattering solution to the problem  (\ref{eq1})-(\ref{rc}).
Then $u^\infty(\hat{x},d)=u^\infty(-d,-\hat{x})$ for all $\hat{x}\in{\Sp^1_+},\;d\in{\Sp^1_-}$.
\end{lemma}

\begin{proof}
The reciprocity relation of the far-field pattern has been proved in \cite{CK13}
for the case of bounded obstacles (see Theorem 3.15 in \cite{CK13}).
For the case of locally rough surfaces, the result can be proved similarly with minor
modifications in conjunction with the integral equation method in \cite{ZZ13}.
\end{proof}

We are now in a position to study the properties of $U_i$ $(i=1,2,3)$ and $W_i$ $(i=1,2,3,4)$.

\begin{lemma}\label{le4}
For arbitrarily fixed $z\in\R^2$ and for all $x\in{\R^2_+}$ with $|x|$ large enough, we have
\be\label{eq9}
&&|U_1(x,z)|\le {C}{|x|^{-1/2}},\\ \label{eq10}
&&|U_i(x,z)|\le {C(1+|z|)}{|x|^{-1/2}},\quad i=2,3,\\ \label{eq11}
&&|W_j(x,z)|\le {C}{|x|^{-1/2}},\quad j=2,3,
\en
where $C>0$ is a constant independent of $x$ and $z$.
\end{lemma}

\begin{proof}
First, the estimates (\ref{eq9}) and (\ref{eq11}) follows easily from Lemma \ref{le3}.

We now prove the estimate of $U_i(x,z)$, $i=2,3$, in (\ref{eq10}). To this end,
define $f_z(d):=-e^{-ikz\cdot d}$. Then
\ben
U_2(x,z)=\int_{\Sp^1_-}e^{ikx\cdot d'}f_z(d)ds(d).
\enn
Apply Lemma \ref{le2} to obtain that
\ben
|U_2(x,z)|\le {C\|f_z(\cdot)\|_{C^1(\ov{\Sp^1_-})}}{|x|^{-1/2}}\le {C(1+|z|)}{|x|^{-1/2}}
\enn
for $x\in{\R^2_+}$ with $|x|$ large enough. The estimate for $U_3(x,z)$ can be obtained
similarly. The proof is thus complete.
\end{proof}

\begin{lemma}\label{lem12}
For arbitrarily fixed $z\in\R^2$ and for all $x\in{\R^2_+}$ with $|x|$ large enough, we have
\be\label{eq101}
&&|W_1(x,z)|\le {C(1+|z|)}{|x|^{-1}},\\ \label{eq12}
&&|W_4(x,z)|\le {C}{|x|^{-1}}.
\en
Here, $C>0$ is a constant independent of $x$ and $z$.
\end{lemma}

\begin{proof}
We first consider $W_1(x,z)$. From Lemma \ref{le3} it follows that for $x\in{\R^2_+}$ with $|x|$
large enough,
\be\label{eq99}
W_1(x,z)=\frac{e^{-ik|x|}}{|x|^{1/2}}\int_{\Sp^1_-}e^{2ikx\cdot d}\ov{u^\infty(\hat{x},d)}
e^{-ikz\cdot d}ds(d)+W_{1,Res}(x,z),
\en
where
\ben
W_{1,Res}(x,z):=\int_{\Sp^1_-}e^{2ikx\cdot d}\ov{u_{Res}(x,d)}e^{-ikz\cdot d}ds(d)
\enn
with
\be\label{eq100}
|W_{1,Res}(x,z)|\le {C}{|x|^{-3/2}}.
\en
Now define
\ben
F(x):=\int_{\Sp^1_-}e^{2ikx\cdot d} f_z(\hat{x},d)ds(d)
\enn
with
\ben
f_z(\hat{x},d):=\ol{u^\infty(\hat{x},d)}e^{-ikz\cdot d},\quad\hat{x}\in\Sp^1_+,\;d\in\Sp^1_-.
\enn
Then, by Lemmas \ref{le3}, \ref{le2} and \ref{lem13} we deduce that
for $x\in{\R^2_+}$ with $|x|$ large enough,
\be\label{eq102}
|F(x)|&\le& C{\|f_z(\hat{x},\cdot)\|_{C^1(\ov{\Sp^1_-})}}{|x|^{-1/2}}\no\\
&\le& C{(1+|z|)\|u^\infty(\hat{x},\cdot)\|_{C^1(\ov{\Sp^1_-})}}{|x|^{-1/2}}
\le C{(1+|z|)}{|x|^{-\half}}.
\en
Thus, (\ref{eq101}) follows immediately from (\ref{eq99}), (\ref{eq100}) and (\ref{eq102}).

We now consider $W_4(x,z)$. By Lemma \ref{le3} we know that $|u^s(x,d)|^2\le C/{|x|}$
for $x\in{\R^2_+}$ with $|x|$ large enough and $d\in\Sp^1_-$. Thus,
(\ref{eq12}) follows from the definition of $W_4(x,z)$.
\end{proof}

\begin{lemma}\label{lem11}
For arbitrarily fixed $z\in\R^2$ and for $R>0$ large enough we have
\ben
&&\left|\int_{\pa B^+_R} U(x,z)\ov{W_j(x,z)}dx\right|\le C{\big(1+|z|\big)^2}{R^{-1/2}},\\
&&\int_{\pa B^+_R}|W_j(x,z)|^2dx\le C{(1+|z|)^2}{R^{-1}}
\enn
for $j=1,4$. Here, $C>0$ is a constant independent of $R$ and $z$.
\end{lemma}

\begin{proof}
From Lemmas \ref{le4} and \ref{lem12} it follows that for $j=1,4$ and $R>0$ large enough,
\ben
&&\left|\int_{\pa B^+_R} U(x,z)\ov{W_j(x,z)}dx\right|
\le C\int_{\pa B^+_R}\frac{1+|z|}{R^{1/2}}\cdot\frac{1+|z|}{R}dx\le C\frac{(1+|z|)^2}{R^{1/2}},\\
&&\left|\int_{\pa B^+_R}|W_j(x,z)|^2dx\right|\le C\int_{\pa B^+_R}\left(\frac{1+|z|}{R}\right)^2dx
\le C\frac{(1+|z|)^2}{R}.
\enn
The proof is thus complete.
\end{proof}

\begin{lemma}\label{lem10}
For arbitrarily fixed $z\in\R^2$ and for $R>0$ large enough we have
\ben
\sum^3_{i=1}\left|\int_{\pa B^+_R} U_i(x,z)\ov{W_j(x,z)}dx\right|
+\sum^4_{i=1}\left|\int_{\pa B^+_R} W_i(x,z)\ov{W_j(x,z)}dx\right|\leq C\frac{(1+|z|)^2}{R^{1/3}}
\enn
for $j=2,3$, where $C>0$ is a constant independent of $R$ and $z$.
\end{lemma}

\begin{proof}
From Lemma \ref{le3} it is easy to derive that for $x\in\R^2_+$ with $|x|$ large enough,
\be\label{eq82}
&&W_2(x,z)=-\frac{e^{-ik|x|}}{|x|^{1/2}}\int_{\Sp^1_-}e^{2i k|x|\hat{x}_1\cdot d_1}
\ov{u^\infty(\hat{x},d)}e^{-ikz\cdot d}ds(d)+W_{2,Res}(x,z),\\ \label{eq83}
&&W_3(x,z)=-\frac{e^{ik|x|}}{|x|^{1/2}}\int_{\Sp^1_-}e^{2i k|x|\hat{x}_2\cdot d_2}
u^\infty(\hat{x},d)e^{-ik z\cdot d}ds(d)+W_{3,Res}(x,z),
\en
where
\ben
&&W_{2,Res}(x,z):=\int_{\Sp^1_-}e^{2i k|x|\hat{x}_1\cdot d_1}\ol{u^s_{Res}(x,d)}e^{-ikz\cdot d}ds(d),\\
&&W_{3,Res}(x,z):=\int_{\Sp^1_-}e^{2i k|x|\hat{x}_2\cdot d_2} u^s_{Res}(x,d)e^{-ik z\cdot d}ds(d)
\enn
with
\be\label{eq103}
\left|W_{j,Res}(x,z)\right|\le {C}{|x|^{-3/2}},\quad j=2,3.
\en
By Lemmas \ref{le4} and \ref{lem12} we obtain that
\be\label{eq84}
\sum\limits^3_{i=1}\left|U_i(x,z)\right|+\sum\limits^4_{i=1}
\left|W_i(x,z)\right|\leq\frac{C(1+|z|)}{|x|^{1/2}},\quad |x|\rightarrow +\infty.
\en
Now, let $x=|x|\hat{x}=|x|(\cos\theta_{\hat{x}},\sin\theta_{\hat{x}})$,
$d=(\cos\theta_d,\sin\theta_d)$ with $\theta_{\hat{x}}\in[0,\pi]$, $\theta_d\in[\pi,2\pi]$, and
define $\wid{f}_z(\theta_{\hat{x}},\theta_d):=\ov{u^\infty(\hat{x},d)}e^{-ik z\cdot d}$,
$\wid{g}_z(\theta_{\hat{x}},\theta_d):=u^\infty(\hat{x},d)e^{-ik z\cdot d}$
for $\theta_{\hat{x}}\in [0,\pi]$ and $\theta_d\in[\pi,2\pi]$.
Then it follows from (\ref{eq82}), (\ref{eq83}), (\ref{eq103}) and (\ref{eq84}) that
\be\label{eq85}
&&\sum\limits^3_{i=1}\left|\int_{\pa B^+_R} U_i(x,z)\ol{W_2(x,z)} dx\right|
+\sum^4_{i=1}\left|\int_{\pa B^+_R} W_i(x,z)\ol{W_2(x,z)}dx\right|\no\\
&&\qquad\;\;\le C(1+|z|)\left(\int_{\Sp^1_+}\left|\int_{\Sp^1_-} e^{2i kR\hat{x}_1\cdot d_1}
\ov{u^\infty(\hat{x},d)}e^{-ikz\cdot d}ds(d)\right|ds(\hat{x})+\frac{C}{R}\right)\no\\
&&\qquad\;\;=C(1+|z|)\left(\int^\pi_0\left|\int^{2\pi}_{\pi}e^{2i k R\cos\theta_{\hat{x}}\cos\theta_d}
\wid{f}_z(\theta_{\hat{x}},\theta_d)d{\theta_d}\right|d{\theta_{\hat{x}}} +\frac{C}{R}\right),\\ \label{eq86}
&&\sum\limits^3_{i=1}\left|\int_{\pa B^+_R} U_i(x,z)\ol{W_3(x,z)}dx\right|
+\sum^4_{i=1}\left|\int_{\pa B^+_R}W_i(x,z)\ov{W_3(x,z)}dx\right|\no\\
&&\qquad\;\;\le C(1+|z|)\left(\int_{\Sp^1_+}\left|\int_{\Sp^1_-} e^{2ikR\hat{x}_2\cdot d_2}
u^\infty(\hat{x},d) e^{-ik z\cdot d} ds(d)\right|ds(\hat{x})+\frac{C}{R}\right)\no\\
&&\qquad\;\;=C(1+|z|)\left(\int^\pi_0\left|\int^{2\pi}_{\pi} e^{2ikR\sin\theta_{\hat{x}}\sin\theta_d}
\wid{g}_z(\theta_{\hat{x}},\theta_d)d\theta_d\right|d\theta_{\hat{x}}+\frac{C}{R}\right).
\en
Let $\vep>0$ be small enough such that $\sin\vep\geq\vep/2$ and let $R$ be large enough.
Define
$$
\wid{w}_z(R,\theta_{\hat{x}}):=\int^{2\pi}_{\pi}e^{2ikR\cos\theta_{\hat{x}}
\cos\theta_d}\wid{f}_z(\theta_{\hat{x}},\theta_d)d{\theta_d}
$$
for $\theta_{\hat{x}}\in[0,\pi]$. Then, by Lemma \ref{le3} we have
\be\label{eq104}
&&\int_0^{\pi}|\wid{w}_z(R,\theta_{\hat{x}})|d{\theta_{\hat{x}}}\no\\
&&\;\;=\int_{[0,\frac{\pi}{2}-\vep]\cup[{\pi}/{2}+\vep,\pi]}|\wid{w}_z(R,\theta_{\hat{x}})|d{\theta_{\hat{x}}}
+\int_{[{\pi}/{2}-\vep,{\pi}/{2}+\vep]}|\wid{w}_z(R,\theta_{\hat{x}})|d{\theta_{\hat{x}}}\no\\
&&\;\;\le C\vep + \int_{[0,{\pi}/{2}-\vep]\cup[{\pi}/{2}+\vep,\pi]}
|\wid{w}_z(R,\theta_{\hat{x}})|d{\theta_{\hat{x}}}\no\\
&&\;\;\le C\vep+\int_{[0,{\pi}/{2}-\vep]\cup[{\pi}/{2}+\vep,\pi]}
\left|\int_{[\pi+\vep,2\pi-\vep]}e^{2ikR\cos\theta_{\hat{x}}\cos\theta_{d}}
\wid{f}_z(\theta_{\hat{x}},\theta_d)d{\theta_d}\right|d{\theta_{\hat{x}}}\no\\
&&\;\;\;+\int_{[0,{\pi}/{2}-\vep]\cup[{\pi}/{2}+\vep,\pi]}
\left|\int_{[\pi,\pi+\vep]\cup[2\pi-\vep,2\pi]}e^{2ikR\cos\theta_{\hat{x}}\cos\theta_{d}}
\wid{f}_z(\theta_{\hat{x}},\theta_d)d{\theta_d}\right|d{\theta_{\hat{x}}}\no\\
&&\;\;\le C\vep+\int_{[0,{\pi}/{2}-\vep]\cup[{\pi}/{2}+\vep,\pi]}
\left|\int_{[\pi+\vep,2\pi-\vep]}e^{2ikR\cos\theta_{\hat{x}}\cos\theta_{d}}
\wid{f}_z(\theta_{\hat{x}},\theta_d)d{\theta_d}\right|d{\theta_{\hat{x}}}.\;\;
\en
Let $u_{\theta_{\hat{x}}}(\theta_d)={4}{\vep^{-2}}\cos\theta_{\hat{x}}\cos\theta_d$.
Then it is easy to see that $u'_{\theta_{\hat{x}}}(\theta_d)=-{4}{\vep^{-2}}
\cos\theta_{\hat{x}}\sin\theta_d$, and so we obtain that for
$\theta_{\hat{x}}\in[0,{\pi}/{2}-\vep]\cup[{\pi}/{2}+\vep,\pi]$ and $\theta_d\in[\pi+\vep,2\pi-\vep]$,
\ben
|u'_{\theta_{\hat{x}}}(\theta_d)|=\frac{4}{\vep^2}|\cos\theta_{\hat{x}}|
|\sin\theta_d|=\frac{4}{\vep^2}\left|\sin(\theta_{\hat{x}}-\pi/2)\right|
|\sin\theta_d|\ge\frac{4}{\vep^2}\sin^2\vep\geq \frac{4}{\vep^2}\left(\frac{\vep}{2}\right)^2=1
\enn
and $u^{\prime\prime}_{\theta_{\hat{x}}}(\theta_d)=-4\vep^{-2}\cos{\theta_{\hat{x}}}\cos\theta_d$
is monotone for $\theta_d\in[\pi+\vep,2\pi-\vep]$.
Thus we can apply Lemmas \ref{le3}, \ref{le1} and \ref{lem13} to obtain that for
$\theta_{\hat{x}}\in[0,{\pi}/{2}-\vep]\cup[\frac{\pi}{2}+\vep,\pi]$
\be\label{eq105}
&&\left|\int_{[\pi+\vep,2\pi-\vep]}e^{2ikR\cos\theta_{\hat{x}}\cos\theta_d}
\wid{f}_z(\theta_{\hat{x}},\theta_d)d{\theta_d}\right|\no\\
&&\quad=\left|\int_{[\pi+\vep,2\pi-\vep]}e^{(ikR\vep^2/2)u_{\theta_{\hat{x}}}(\theta_d)}
\wid{f}_z(\theta_{\hat{x}},\theta_d)d{\theta_d}\right|\no\\
&&\quad\le\frac{C}{R\vep^2}\|\wid{f}_z(\theta_{\hat{x}},\cdot)\|_{C^1[\pi+\vep,2\pi-\vep]}\no\\
&&\quad\le\frac{C(1+|z|)\|u^\infty(\hat{x},\cdot)\|_{C^1(\ol{\Sp^1_-})}}{R\vep^2}
\le \frac{C(1+|z|)}{R\vep^2}.
\en
Combining (\ref{eq104}) and (\ref{eq105}) and then taking $\vep=R^{-1/3}$ give
\be\label{eq87}
\int^\pi_0\left|\wid{w}_z(R,\theta_{\hat{x}})\right|d{\theta_{\hat{x}}}
\leq C\vep+\frac{C(1+|z|)}{R\vep^2}\leq C\frac{1+|z|}{R^{1/3}}.
\en

Now, define
$$
\wid{v}_z(R,\theta_{\hat{x}}):=\int^{2\pi}_{\pi}e^{2ikR\sin\theta_{\hat{x}}\sin\theta_d}
\wid{g}_z(\theta_{\hat{x}},\theta_d)d{\theta_d}.
$$
Then it follows from Lemma \ref{le3} that
{\small
\be\label{eq106}
\int^{\pi}_0|\wid{v}_z(R,\theta_{\hat{x}})|d\theta_{\hat{x}}
&=&\int_{[\vep,\pi-\vep]}|\wid{v}_z(R,\theta_{\hat{x}})|d\theta_{\hat{x}}
+\int_{[0,\vep]\cup[\pi-\vep,\pi]}|\wid{v}_z(R,\theta_{\hat{x}})|d\theta_{\hat{x}}\no\\
&\le& C\vep + \int_{[\vep,\pi-\vep]}|\wid{v}_z(R,\theta_{\hat{x}})|d\theta_{\hat{x}}\no\\
&\le& C\vep+\int_{[\vep,\pi-\vep]}\left|\int_{[\pi,3\pi/2-\vep]\cup
    [3\pi/2+\vep,2\pi]}e^{2ikR\sin\theta_{\hat{x}}\sin\theta_d}
    \wid{g}_z(\theta_{\hat{x}},\theta_d)d\theta_d\right|d\theta_{\hat{x}}\no\\
&&\;\;+\int_{[\vep,\pi-\vep]}\left|\int_{[3\pi/2-\vep,3\pi/2+\vep]}
   e^{2ikR\sin\theta_{\hat{x}}\sin\theta_d}
   \wid{g}_z(\theta_{\hat{x}},\theta_d)d\theta_d\right|d\theta_{\hat{x}}\no\\
&\le& C\vep+\int_{[\vep,\pi-\vep]}\left|\int_{[\pi,3\pi/2-\vep]\cup[3\pi/2+\vep,2\pi]}
  e^{2ikR\sin\theta_{\hat{x}}\sin\theta_{\theta_d}}
  \wid{g}_z(\theta_{\hat{x}},\theta_d)d\theta_d\right|d\theta_{\hat{x}}.\no\\
\en
}
Let $v_{\theta_{\hat{x}}}(\theta_d)=4\vep^{-2}\sin\theta_{\hat{x}}\sin\theta_d$.
It is easy to see that $v'_{\theta_{\hat{x}}}(\theta_d)=4\vep^{-2}\sin\theta_{\hat{x}}\cos\theta_d$,
and thus we have that for $\theta_{\hat{x}}\in[\vep,\pi-\vep]$ and
$\theta_d\in[\pi,3\pi/2-\vep]\cup[3\pi/2+\vep,2\pi]$,
{\small
\ben
|v'_{\theta_{\hat{x}}}(\theta_d)|=\frac{4}{\vep^2}|\sin\theta_{{\hat{x}}}||\cos\theta_{d}|
=\frac{4}{\vep^2}|\sin\theta_{\hat{x}}||\sin(\theta_d-3\pi/2)|
\ge\frac{4}{\vep^2}\sin\vep\sin\vep\ge\frac{4}{\vep^2}\left(\frac{\vep}{2}\right)^2=1
\enn
}
and $v^{\prime\prime}_{\theta_{\hat{x}}}(\theta_d)=-4\vep^{-2}\sin\theta_{\hat{x}}\sin{\theta_d}$
is monotone for $\theta_d\in[\pi,3\pi/2-\vep]$ and for $\theta_d\in[3\pi/2+\vep,2\pi]$.
Then, by Lemmas \ref{le3}, \ref{le1} and \ref{lem13} we find that for $\theta_{\hat{x}}\in[\vep,\pi-\vep]$,
\be\label{eq107}
&&\left|\int_{[\pi,3\pi/2-\vep]\cup[3\pi/2+\vep,2\pi]}e^{2ikR\sin\theta_{\hat{x}}\sin\theta_d}
\wid{g}_z(\theta_{\hat{x}},\theta_d)d\theta_d\right|\no\\
&&\;\;=\left|\int_{[\pi,3\pi/2-\vep]\cup[3\pi/2+\vep,2\pi]}e^{(ikR\vep^2/2)v_{\theta_{\hat{x}}}(\theta_d)}
\wid{g}_z(\theta_{\hat{x}},\theta_d)d\theta_d\right|\no\\
&&\;\;\le\frac{C}{R\vep^2}\left\|\wid{g}_z(\theta_{\hat{x}},\cdot)
  \right\|_{C^1([\pi,3\pi/2-\vep]\cup[3\pi/2+\vep,2\pi])}\no\\
&&\;\;\le\frac{C(1+|z|)}{R\vep^2}\|u^\infty(\hat{x},\cdot)\|_{C^1(\ov{S^1_-})}
\le\frac{C(1+|z|)}{R\vep^2}.
\en
Combining (\ref{eq106}) and (\ref{eq107}) and taking $\vep=R^{-1/3}$ yield
\be\label{eq88}
\int^\pi_0|\wid{v}_z(R,\theta_{\hat{x}})|d\theta_{\hat{x}}\leq C\vep+ \frac{C (1+|z|)}{R\vep^2}
\leq C\frac{1+|z|}{R^{1/3}}.
\en
Finally, combining (\ref{eq85}), (\ref{eq86}), (\ref{eq87}) and (\ref{eq88}) gives
\ben
&&\sum^3_{i=1}\left|\int_{\pa B^+_R}U_i(x,z)\ov{W_j(x,z)}dx\right|
+\sum^4_{i=1}\left|\int_{\pa B^+_R}W_i(x,z)\ov{W_j(x,z)}dx\right|\\
&&\quad\le C(1+|z|)\left(C\frac{1+|z|}{R^{1/3}}+\frac{C}{R}\right)
\le C\frac{(1+|z|)^2}{R^{1/3}},\quad j=2,3.
\enn
The proof is thus completed.
\end{proof}

For $z\in\R^2$ define the function
\be\label{eq15}
F(R,z):=\int_{\pa B^+_R}\left|U(x,z)\right|^2dx,
\en
where $U(x,z)$ is given in (\ref{eq16}). The following lemma gives the properties of $F(R,z)$ for
sufficiently large $R$. The proof of this lemma is mainly based on the method of stationary phase
and will be presented in Appendix A.

\begin{lemma}\label{lem5}
For $z\in\R^2$ and $R>0$ we have $F(R,z)=F_0(z)+F_{0,Res}(R,z)$, where
{\small
\be\label{eq27}
F_0(z):=\int_{\Sp^1_+}\left|\int_{\Sp^1_-}u^\infty(\hat{x},d)e^{-ik z\cdot d}ds(d)
-\left(\frac{2\pi}{k}\right)^{1/2}e^{-\frac{\pi}{4}i}
\left(e^{-ik\hat{x}\cdot z'}+e^{-ik\hat{x}\cdot z}\right)\right|^2ds(\hat{x})\no\\
\en
}
which is independent on $R$, and $F_{0,Res}(R,z)$ satisfies the estimate
\be\label{eq76}
\left|F_{0,Res}(R,z)\right|\le C\frac{(1+|z|)^4}{R^{1/4}}
\en
for sufficiently large $R$. Here, $C>0$ is a constant independent of $R$ and $z$.
\end{lemma}

From (\ref{eq89}) it follows that
{\small
\ben
I^{Phaseless}(z)=\int_{\pa B^+_R}|U(x,z)|^2dx+2\Rt\int_{\pa B^+_R}U(x,z)\ov{W(x,z)}dx
+\int_{\pa B^+_R}|W(x,z)|^2dx.
\enn
}
Define
\ben
F_{Res}(R,z):=2\textrm{Re}\int_{\pa B^+_R}U(x,z)\ov{W(x,z)}dx+\int_{\pa B^+_R}|W(x,z)|^2dx.
\enn
Then, by Lemmas \ref{lem11}, \ref{lem10} and \ref{lem5} we obtain the main theorem of this section.

\begin{theorem}\label{thm1}
For $z\in\R^2$ and $R>0$ we have
\be\label{eq90}
I^{Phaseless}(z)=F(R,z)+F_{Res}(R,z),
\en
where $F(R,z)$ is defined in (\ref{eq15}) and $F_{Res}(R,z)$ satisfies the estimate
\be\label{eq91}
\left|F_{Res}(R,z)\right|\leq C\frac{(1+|z|)^2}{R^{1/3}}
\en
for $R$ large enough and $C>0$ independent of $R$ and $z$.
Further, $F(R,z)=F_0(z)+F_{0,Res}(R,z)$, where $F_0(z)$ is defined in (\ref{eq27})
and $F_{0,Res}(R,z)$ satisfies the estimate (\ref{eq76}).
\end{theorem}

With the help of the above analysis, we now study properties of the imaging function $I^{Phaseless}(z)$,
$z\in\R^2$. Let $K$ be a bounded domain which contains the local perturbation $\G_p$ of
the locally rough surface $\G$. From Theorem \ref{thm1} it is easy to see that
if $R$ is large enough then $I^{Phaseless}(z)\approx F(R,z)$ for $z\in K$ with $F(R,z)$ given by (\ref{eq15}).
Thus the imaging function $I^{Phaseless}(z)$ is approximately equal to the function $F(R,z)$ for $z\in K$.
Therefore, in what follows, we investigate the properties of the function $F(R,z)$.
We will make use of the theory of scattering by unbounded rough surfaces.
To this end, for $b\in\R$ let $U^+_b=\{x=(x_1,x_2)\in\R^2|x_2> b\}$ and $\G_b=\{x=(x_1,x_2)\in\R^2|x_2= b\}$.
Further, let $BC(\G)$ denote the Banach space of functions which are bounded and continuous on $\G$
with the norm $\|\psi\|_{\infty,\G}:=\sup_{x\in\G}|\psi(x)|$ for $\psi\in BC(\G)$.
Then the problem of scattering by an unbounded, sound-soft, rough surface can be formulated
as the following Dirichlet boundary value problem (see \cite{CZ98,CZ99,ZC03}).

{\bf Dirichlet problem (DP)}: Given $g\in BC(\G)$, determine $u\in C^2(D_+)\cap C(\ov{D_+})$
such that

(i) $u$ satisfies the Helmholtz equation (\ref{eq1});

(ii) $u=g$ on $\G$;

(iii) For some $a\in\R$,
\be\label{eq109}
\sup_{x\in D_+} x^a_2|u(x)|<\infty;
\en

(iv) $u$ satisfies the upward propagating radiation condition (UPRC): for some
$b>h_+:=\sup_{x_1\in\R}h(x_1)$ and $\phi\in L^\infty(\G_b)$,
\be\label{eq108}
u(x)=2\int_{\G_b}\frac{\pa\Phi_k(x,y)}{\pa y_2}\phi(y)ds(y),\quad x\in U^+_b,
\en
where $\Phi_k(x,y):=(i/4)H^{(1)}_0(k|x-y|),x,y\in\R^2,x\neq y$, is the free-space Green's
function for the Helmholtz equation $\Delta u+k^2u=0$ in $\R^2$.

The well-posedness of the problem (DP) has been established in \cite{CZ98,CZ99,ZC03},
using the integral equation method.
The following theorem tells us that for arbitrarily fixed $z\in\R^2$ the function $U(x,z)$
given by (\ref{eq16}) is the unique solution to the Dirichlet problem (DP) with the boundary data
involving the Bessel function of order $0$.

\begin{theorem}\label{thm3}
For arbitrarily fixed $z\in\R^2$, $U(x,z)$ given by (\ref{eq16}) satisfies the Dirichlet problem (DP)
with the boundary data
\be\label{eq20}
g(x)=g_z(x):=-2\pi J_0(k|x-z|),\quad x\in\G,
\en
where $J_0$ is the Bessel function of order $0$.
\end{theorem}

\begin{proof}
Arbitrarily fix $d\in\Sp^1_-$ and define $\wid{u}^s(x,d):=u^s(x,d)-e^{ik x\cdot d'}$ with $x\in D_+$.
From the well-posedness of the scattering problem (\ref{eq1})-(\ref{rc}), it is easily seen
that $\wid{u}^s(x,d)$ satisfies the Helmholtz equation (\ref{eq1}) and the condition (\ref{eq109}).
Since $u^s(x,d)$ satisfies the Sommerfeld radiation condition (\ref{rc}), it follows
from \cite[Theorem 2.9]{CZ98b} that $u^s(x,d)$ also satisfies the UPRC condition (\ref{eq108}).
Further, by \cite[Remark 2.15]{CZ98b} we know that $e^{ik x\cdot d'}$ satisfies the UPRC
condition (\ref{eq108}). As a result, $\wid{u}^s(x,d)$ satisfies the UPRC condition (\ref{eq108}),
and thus apply the boundary condition (\ref{eq2}) to deduce that $\wid{u}^s(x,d)$ is the solution
to the Dirichlet problem (DP) with the boundary data $g(x)=-u^i(x,d)=-e^{ik x\cdot d}$.
Furthermore, by the definition of $U(x,z)$ we see that
\ben
U(x,z)=\int_{\Sp^1_-}\left[\wid{u}^s(x,d)e^{-ik z\cdot d}-e^{ik x\cdot d'-z'\cdot d}\right]ds(d).
\enn
Then, by the Funk-Hecke formula (see, e.g., \cite[Lemma 2.1]{ZZ18}) it is derived that
$U(x,z)=-2\pi J_0(k|x-z|)$, $x\in\G$, and so, $U(x,z)$ satisfies the Dirichlet problem (DP)
with the boundary data given by (\ref{eq20}). The theorem is thus proved.
\end{proof}

Properties of solutions to the Dirichlet problem (DP) with the boundary data $g(x)=aJ_0(k|x-z|)$, $x\in\G$,
for any $a\in\R$ have been investigated in the case when $\G$ is a globally rough surface
(see \cite[Section 3]{LZZ18}). From the discussions in \cite[Section 3]{LZZ18}, it is expected that
for any $x$ in the compact subset of $D_+$ the function $U(x,z)$ given in (\ref{eq16}) will take
a large value when $z\in\G$ and decay as $z$ moves away from $\G$.
As a result, it is expected that for any fixed $R>0$ the function $F(R,z)$ defined in (\ref{eq15})
will take a large value when $z\in\G$ and decay as $z$ moves away from $\G$.
Thus, by Theorem \ref{thm1} we know that for any bounded sampling region $K$
the imaging function $I^{Phaseless}(z)$ will have similar properties as $F(R,z)$ with $z\in K$
if $R$ is large enough, as seen in the numerical experiments presented in the next section.

\begin{remark}\label{re1}{\rm
In the numerical experiments, we measure the phaseless total-field data $|u(x^{(i)},d^{(j)})|$,
$i=1,2,\ldots,M,\;j=1,2,\ldots,N$, where $x^{(i)}=(x^{(i)}_1,x^{(i)}_2)$ and $d^{(j)}=(d^{(j)}_1,d^{(j)}_2)$
are uniformly distributed points on $\pa B^+_R$ and $\Sp^1_-$, respectively.
Accordingly, the imaging function $I^{Phaseless}(z)$ is approximated as
{\small
\be\label{eq26}
&&I^{Phaseless}(z)\no\\
&&\approx\frac{\pi^3 R}{MN^2}\sum\limits^{M}_{i=1}\left|
\sum\limits^{N}_{j=1}\left[\left(\left|u(x^{(i)},d^{(j)})\right|^2-2+e^{2ikx^{(i)}_2d^{(j)}_2}\right)
e^{ik(x^{(i)}-z)\cdot d^{(j)}}-e^{ik\left(x^{(i)'}-z'\right)\cdot d^{(j)}}\right]\right|^2,\no\\
\en
}
where $x^{(i)'}:=(x^{(i)}_1,-x^{(i)}_2)$.
}
\end{remark}

The direct imaging algorithm for our inverse problem can be given in the following algorithm.

\begin{algorithm}\label{al1}
Let $K$ be the sampling region which contains the local perturbation $\G_p$ of the locally rough surface $\G$.

1) Choose $\mathcal{T}_m$ to be a mesh of $K$ and take $R$ to be a large number.

2) Collect the phaseless total-field data $\left|u(x^{(i)},d^{(j)})\right|$,
$i=1,2,\ldots,M,\;j=1,2,\ldots,N$, with $x^{(i)}\in\pa B^+_R$ and $d^{(j)}\in{\Sp^1_-}$,
generated by the incident plane waves $u^i(x,d^{(j)})=e^{i{k}x\cdot d^{(j)}},\;j=1,2,\ldots,N$.

3) For all sampling points $z\in\mathcal{T}_m$, compute the imaging function $I^{Phaseless}(z)$
given in (\ref{eq26}).

4) Locate all those sampling points $z\in\mathcal{T}_m$ such that $I^{Phaseless}(z)$ takes a large value,
which represent the part of the locally rough surface $\G$ in the sampling region $K$.
\end{algorithm}

\begin{remark}\label{rem1} {\rm
Let $K$ be the bounded sampling domain as above.
From Lemma \ref{lem5} it is seen that if $R$ is large enough then $F(R,z)\approx F_0(z)$ for $z\in K$,
and so, by the properties of $F(R,z)$ as discussed above we know that
the function $F_0(z)$ defined in (\ref{eq27}) will be expected to take a large value when $z\in\G$ and
decay as $z$ moves away from $\G$. Based on this, we define $I^{Full}(z):=F_0(z)$ for $z\in\R^2$ to be
the imaging function with the full far-field data $u^\infty(\hat{x},d)$ with $\hat{x}\in\Sp^1_+$
and $d\in\Sp^1_-$. In the numerical experiments presented in the next section, we will show the
imaging results of $I^{Full}(z)$ to compare with those of the imaging function $I^{Phaseless}(z)$.
Therefore, we will take the full far-field measurement data $u^\infty(\hat{x}^{(i)},d^{(j)})$, $i=1,2,\ldots,L,\;j=1,2,\ldots,N$, where $\hat{x}^{(i)}$ and $d^{(j)}$ are uniformly distributed
points on $\Sp^1_+$ and $\Sp^1_-$, respectively.
Accordingly, the imaging function $I^{Full}(z)$ is approximated as
{\small
\ben
&&I^{Full}(z)\\
&&\approx\frac{\pi}{L}\sum\limits^{L}_{i=1}\left|\frac{\pi}{N}
\left(\sum\limits^{N}_{j=1}u^\infty(\hat{x}^{(i)},d^{(j)})e^{-ikz\cdot d^{(j)}}\right)
-\left(\frac{2\pi}{k}\right)^{1/2}e^{-\pi i/4}\left(e^{-ik\hat{x}^{(i)}\cdot z'}
+e^{-ik\hat{x}^{(i)}\cdot z}\right)\right|^2.
\enn
}
The direct imaging algorithm based on the imaging function $I^{Full}(z)$ can be given similarly
as in Algorithm \ref{al1}.
}
\end{remark}

\section{Numerical experiments}\label{sec4}
\setcounter{equation}{0}

In this section, we present several numerical experiments to demonstrate the effectiveness of our imaging
algorithm with the phaseless total-field data. Though the locally rough surface is assumed to be smooth
in the above sections, we will also consider the reconstructed results for the case when
the locally rough surface is piecewise smooth. In addition, in each examples, we will also present
imaging results of the imaging algorithm with full far-field data to compare the reconstruction results
using both the phaseless near-field measurement data and the full far-field measurement data.
To generate the synthetic data, we use the integral equation method proposed in \cite{ZZ13}
to solve the forward scattering problem (\ref{eq1})-(\ref{rc}).
Further, the noisy phaseless near-field data $|u_\delta(x,d)|$, $x\in\pa B^+_R,\;d\in{\Sp^1_-}$,
and the noisy full far-field data $u^\infty_\delta(\hat{x},d)$, $\hat{x}\in\Sp^1_+,\;d\in\Sp^1_-$,
are simulated by
\ben
|u_\delta(x,d)|&=&|u(x,d)|\left(1+\delta\zeta_1\right),\\
u^\infty_\delta(\hat{x},d)&=&u^\infty(\hat{x},d)+\delta\left(\zeta_2+i\zeta_3\right)|u^\infty(\hat{x},d)|,
\enn
where $\delta$ is the noise ratio and $\zeta_1,\zeta_2,\zeta_3$ are the normally distributed random
numbers in $[-1,1]$. In all the figures presented, we use solid line '-' to represents the actual curves.

\textbf{Example 1.} We first investigate the effect of the noise ratio $\delta$ on the imaging results.
The locally rough surface is given by
\ben
h(x_1)=0.1\phi\left(\frac{x_1+0.2}{0.3}\right)-0.08\phi\left(\frac{x_1-0.3}{0.2}\right),
\enn
where
\ben
\phi(x):=\sum^{5}_{j=0}\frac{(-1)^j\left(\begin{array}{c}5 \\ j \end{array}\right)
\left(x+\frac{5}{2}-j\right)^4_+}{4!}\quad\mbox{with}\;\;
x^4_+:=\left\{\begin{array}{ll}x^4 &x\ge0\\ 0 &x<0\end{array}\right.
\enn
is the cubic spline function. The wave number is set to be $k=40$.
We first consider the inverse problem with phaseless near-field data.
We choose the radius of the measurement circle $\pa B^+_R$ to be $R=4$ and the number of both
the measurement points and the incident directions to be the same with $M=N=200$.
Figure \ref{fig1} presents the imaging results of $I^{Phaseless}(z)$ from the measured
phaseless near-field data without noise, with $10\%$ noise, with $20\%$ noise and with $40\%$ noise,
respectively. Next, we consider the inverse problem with full far-field data.
We choose the number of both the measured observation directions and the measured incident directions
to be the same as well with $L=N=100$.
Figure \ref{fig2} presents the imaging results of $I^{Full}(z)$ from the measured full far-field data
without noise, with $10\%$ noise, with $20\%$ noise and with $40\%$ noise, respectively.
As shown by Figures \ref{fig1} and \ref{fig2}, the imaging results given by the imaging function
$I^{Phaseless}(z)$ with phaseless near-field data are good though the imaging results of the
imaging function $I^{Full}(z)$ with full far-field data are better than those of the imaging function
$I^{Phaseless}(z)$ with phaseless near-field data.

\begin{figure}[htbp]
\centering
\subfigure[\textbf{No noise, k=40, R=4}]{
\includegraphics[width=2in]{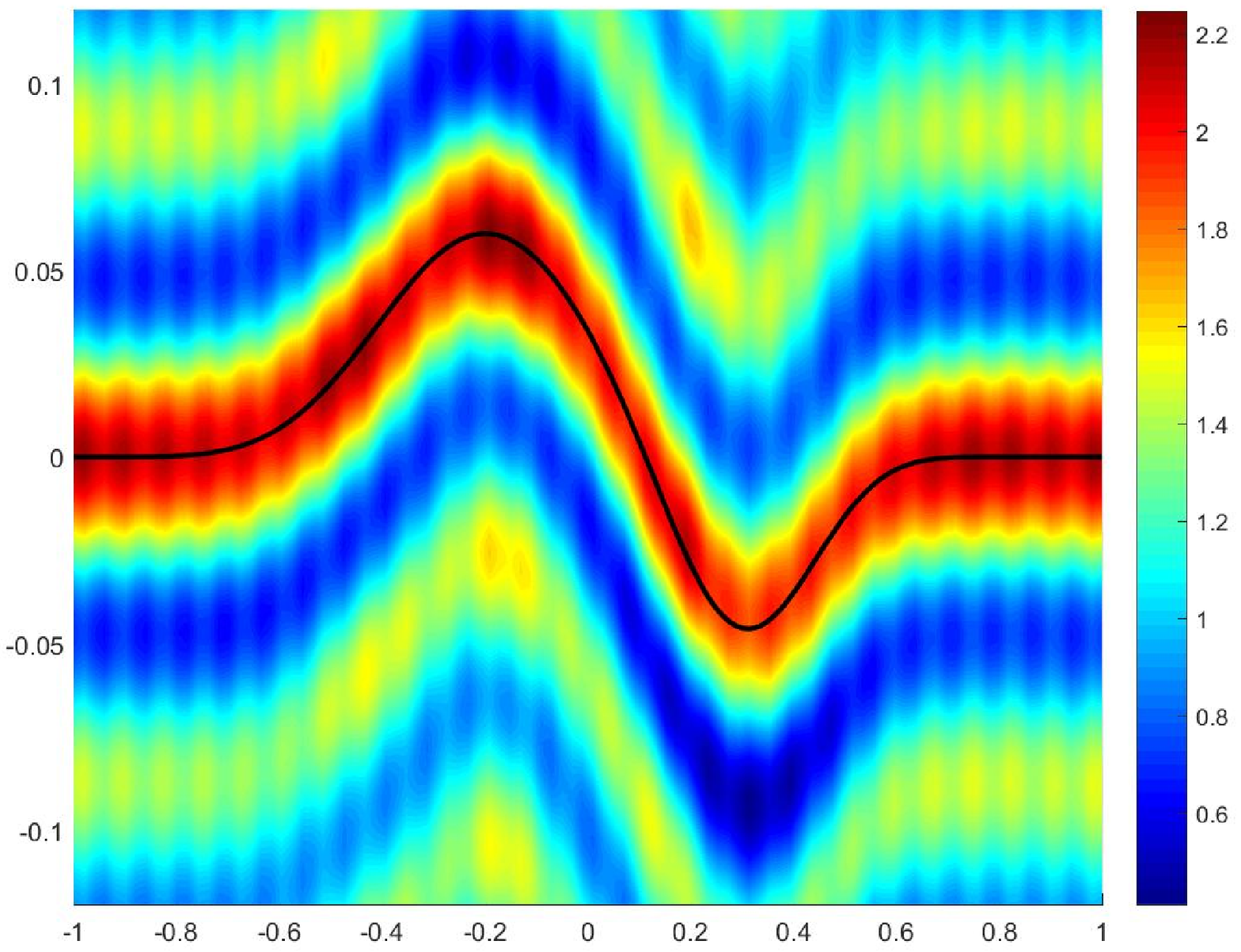}}
\subfigure[\textbf{10\% noise, k=40, R=4}]{
\includegraphics[width=2in]{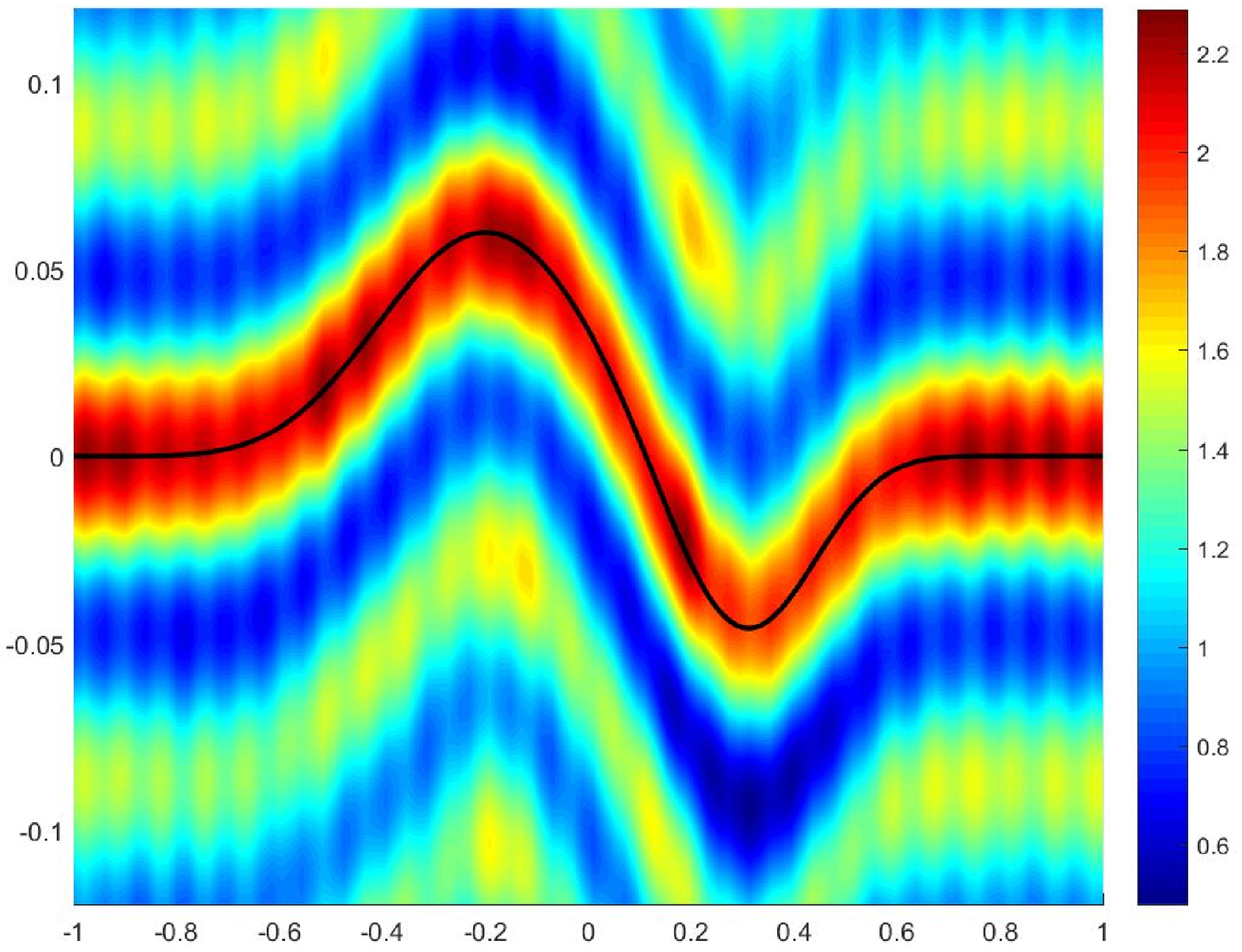}}
\subfigure[\textbf{20\% noise, k=40, R=4}]{
\includegraphics[width=2in]{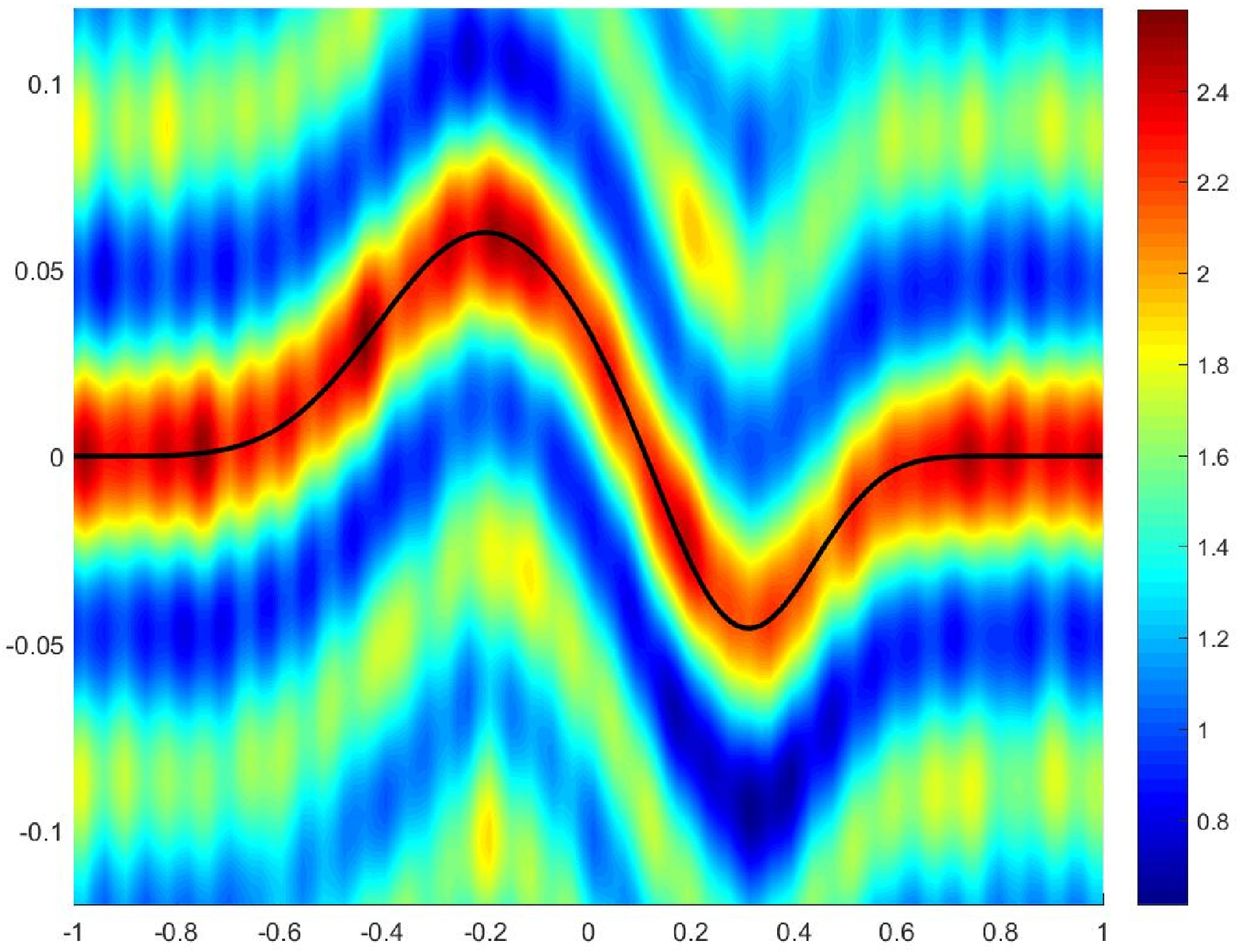}}
\subfigure[\textbf{40\% noise, k=40, R=4}]{
\includegraphics[width=2in]{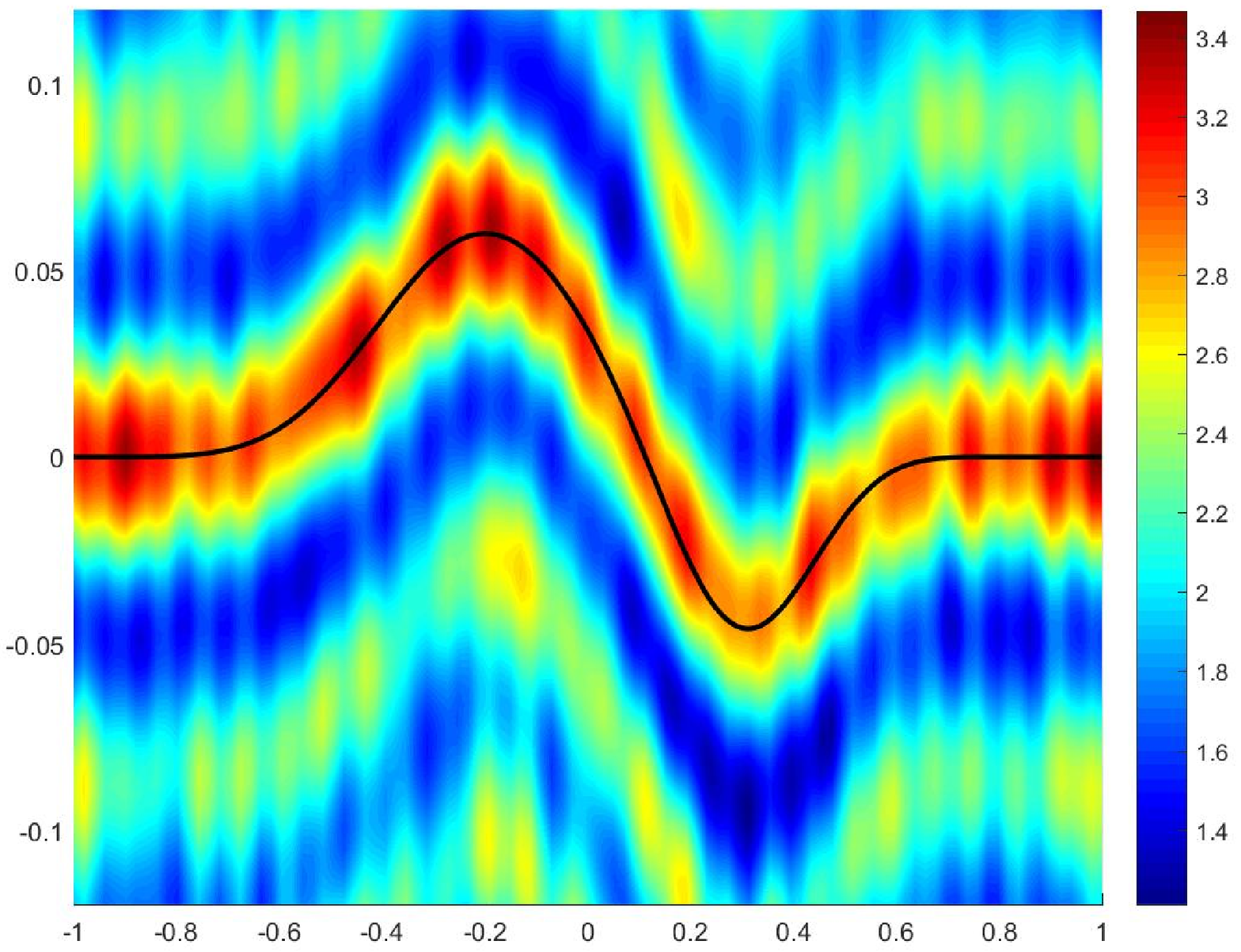}}
\caption{Imaging results of $I^{Phaseless}(z)$ with measured phaseless near-field data,
where the solid line '-' represents the actual curve. The number of both the measured points and the
incident directions is chosen to be the same with $M=N=200$.
}\label{fig1}
\end{figure}

\begin{figure}[htbp]
  \centering
  \subfigure[\textbf{No noise, k=40}]{
    \includegraphics[width=2in]{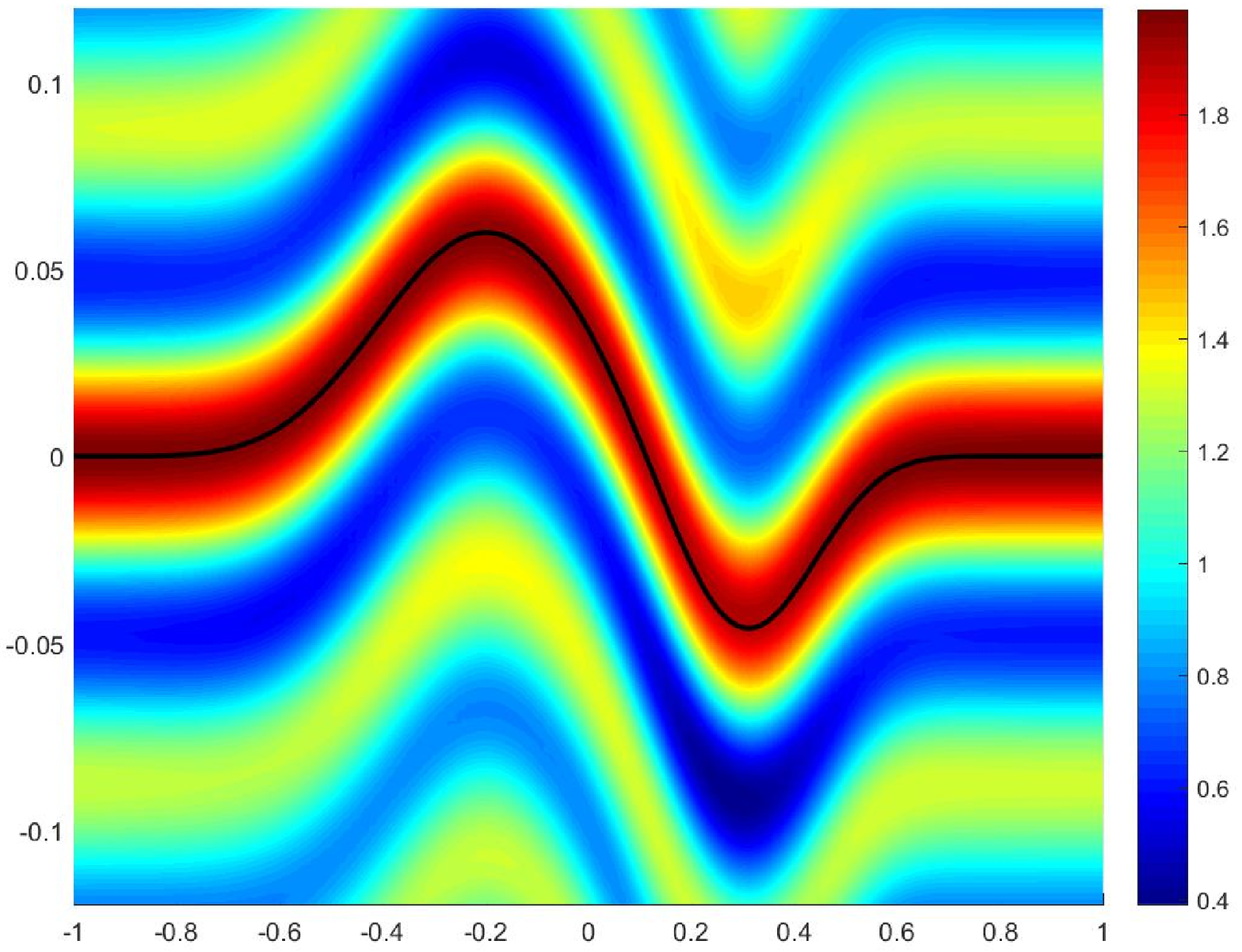}}
  \subfigure[\textbf{10\% noise, k=40}]{
    \includegraphics[width=2in]{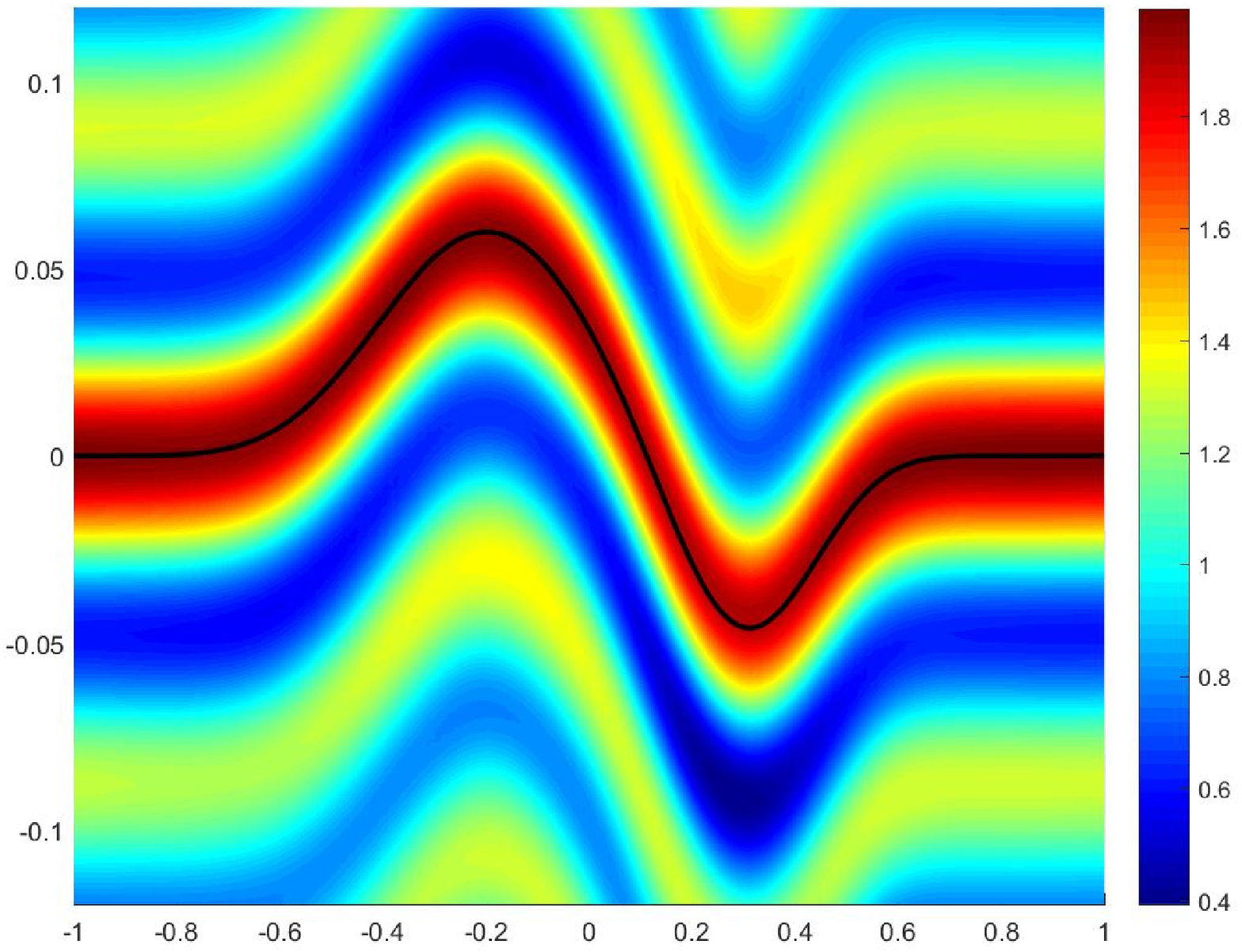}}
  \subfigure[\textbf{20\% noise, k=40}]{
    \includegraphics[width=2in]{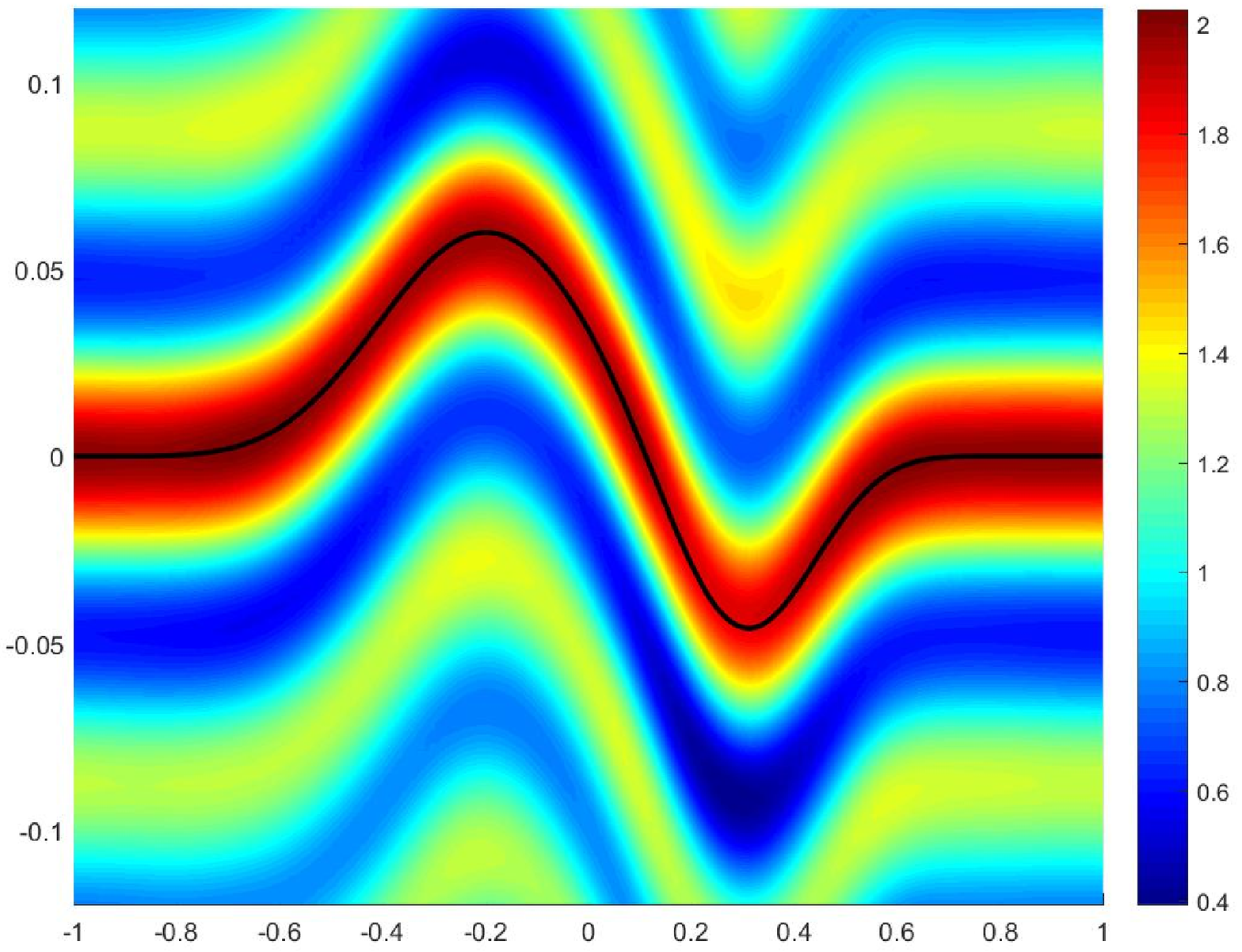}}
  \subfigure[\textbf{40\% noise, k=40}]{
    \includegraphics[width=2in]{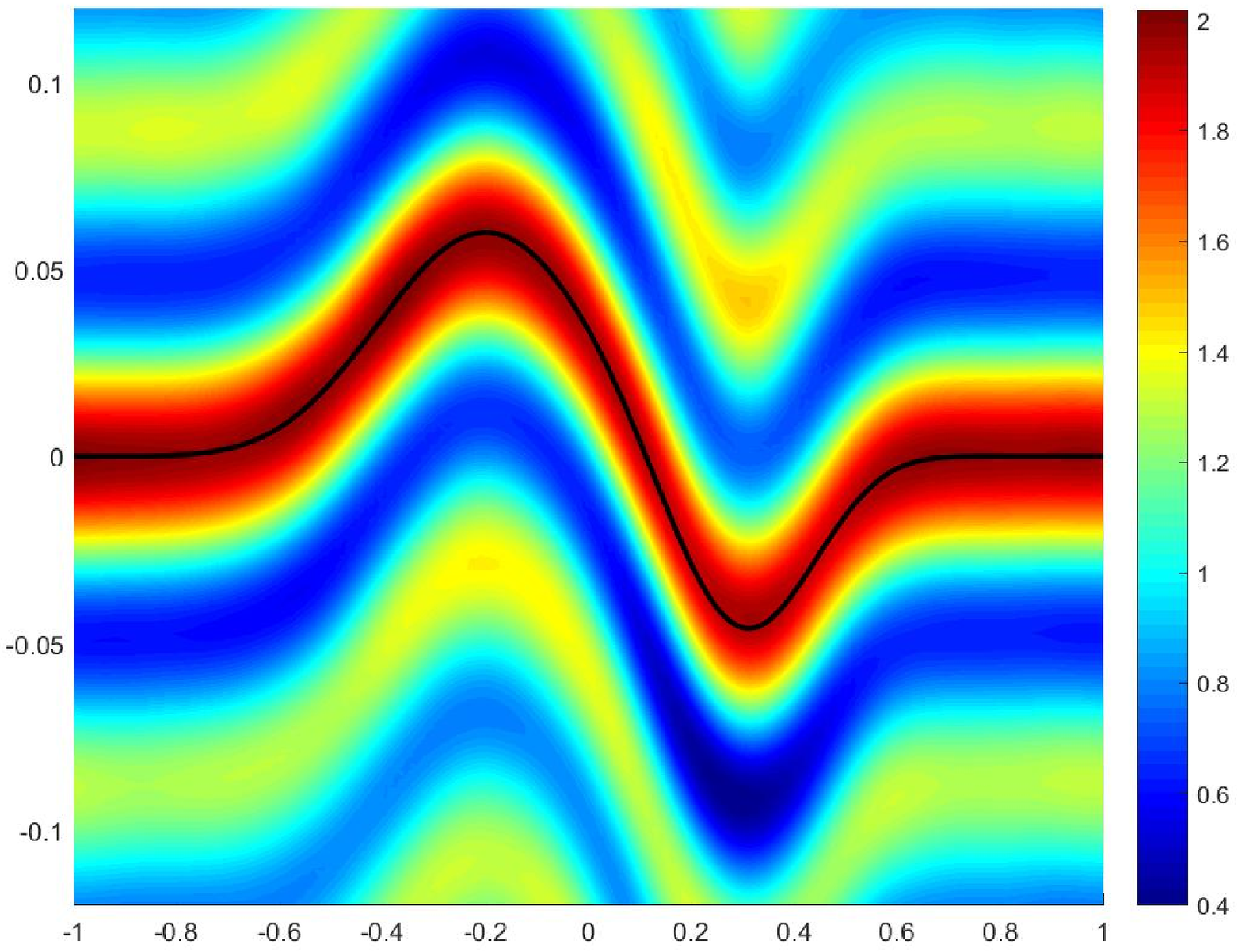}}
\caption{Imaging results of $I^{Full}(z)$ with measured full far-field data, where the solid line '-'
represents the actual curve. The number of both the measured points and the incident directions is
chosen to be the same with $L=N=100$.
}\label{fig2}
\end{figure}

\textbf{Example 2.} We now consider the case when the local perturbation part of the boundary $\G$ is
piecewise linear (the solid line in Figure \ref{fig3}).
We choose the wave number to be $k=80$ and the noise ratio to be $\delta=20\%$.
First consider the inverse problem with phaseless near-field data.
For this case, we investigate the effect of the radius $R$ of the measurement circle
$\pa B^+_R$ on the imaging results.
We choose the number of both the measurement points and the incident directions to be the same with
$M=N=200$. Figures \ref{fig3-a}-\ref{fig3-c} present the imaging results of $I^{Phaseless}(z)$ with
the measurement phaseless near-field data with the radius of the measurement circle
$\pa B^+_R$ to be $R=1.2,\,1.6,\,2$, respectively.
From Figures \ref{fig3-a}-\ref{fig3-c} it is seen that the reconstruction result is getting better with
the radius of the measurement circle getting larger.

Second, we consider the inverse problem for full far-field data.
We choose the numbers of measured directions and incident directions to be $L=N=100$.
Figure \ref{fig3-d} presents the imaging results of $I^{{Full}}$ from the measured
full far-field data.

\begin{figure}[htbp]
\centering
\subfigure[\textbf{20\% noise, k=80, R=1.2}]{\label{fig3-a}
\includegraphics[width=2in]{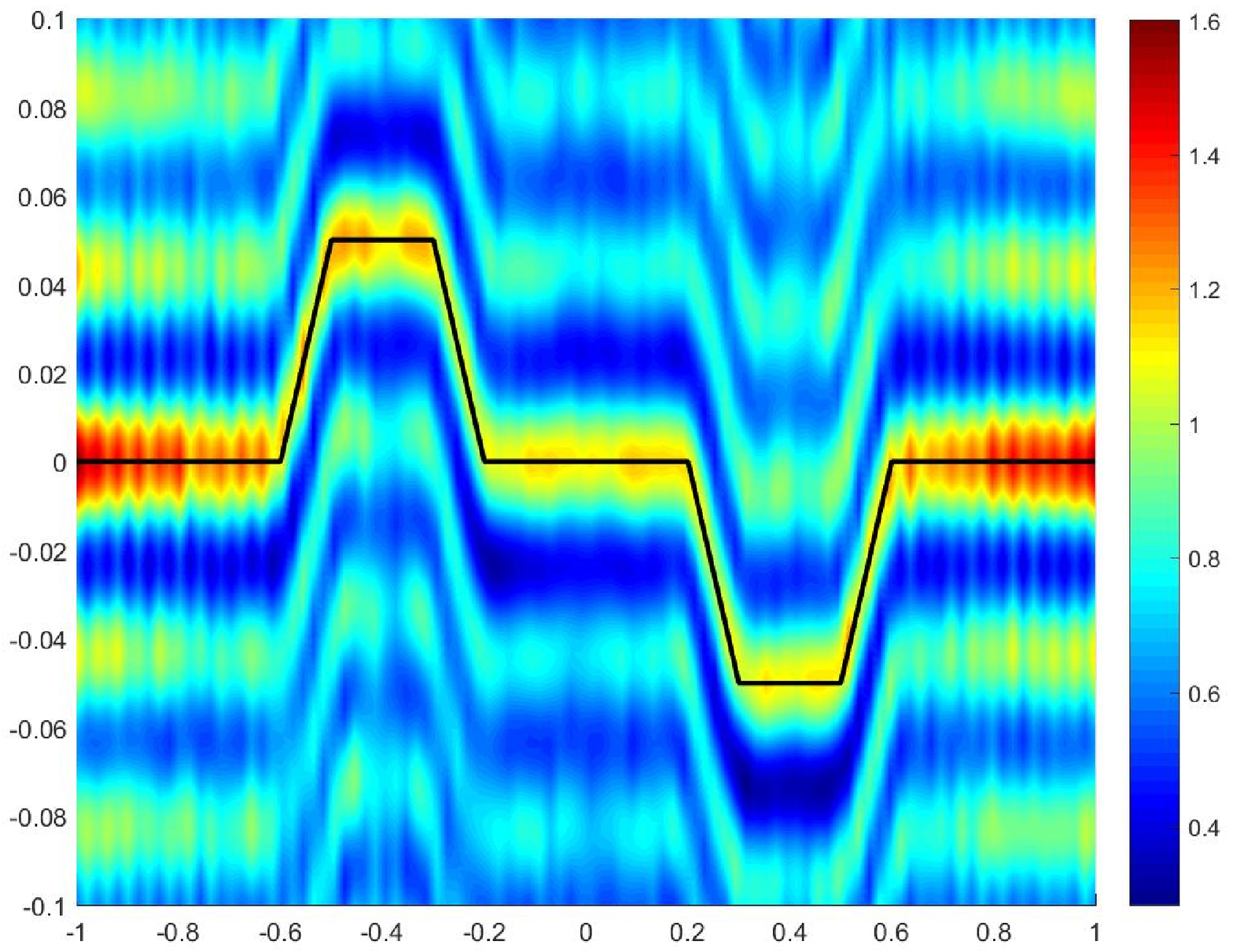}}
\subfigure[\textbf{20\% noise, k=80, R=1.6}]{\label{fig3-b}
\includegraphics[width=2in]{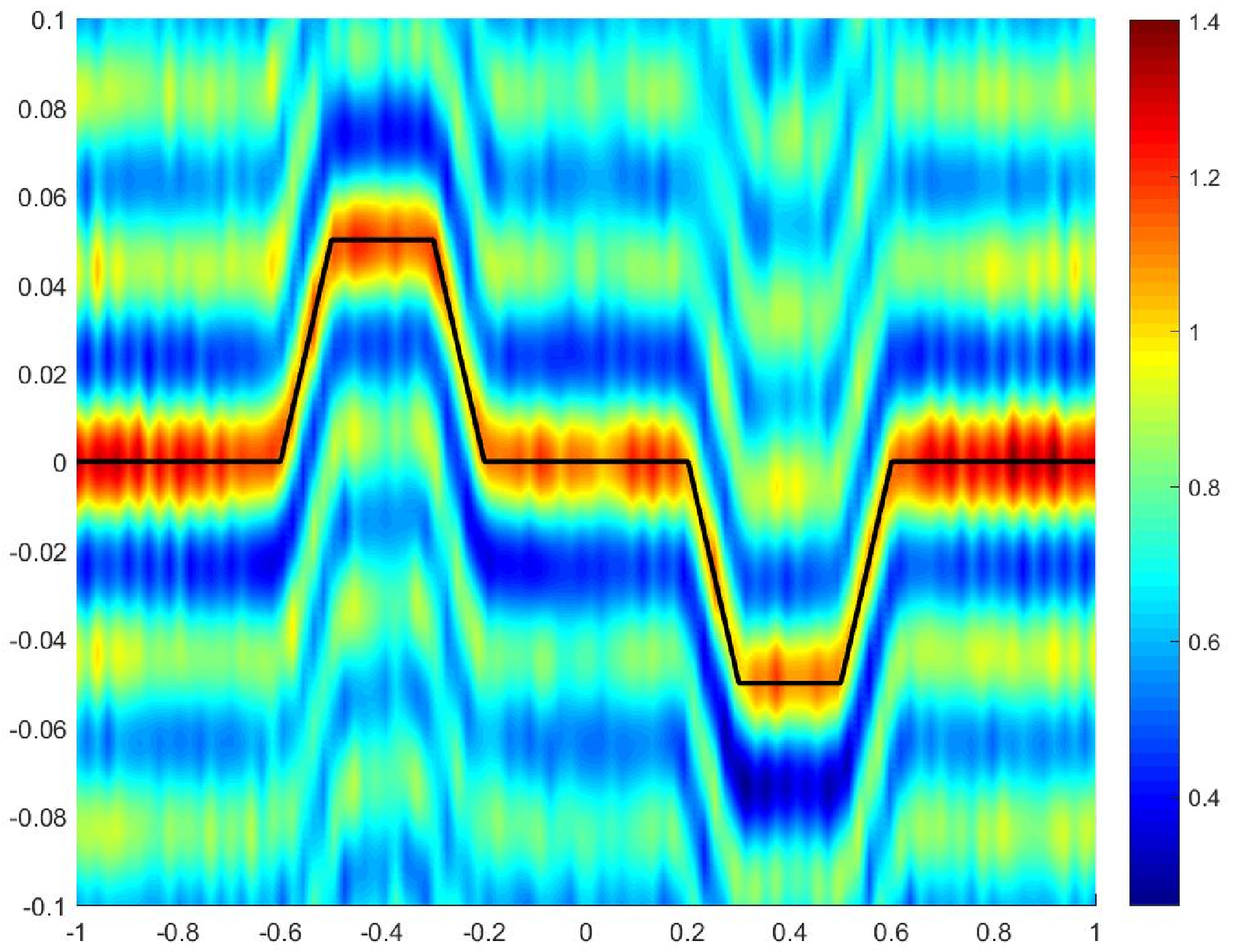}}
\subfigure[\textbf{20\% noise, k=80, R=2}]{\label{fig3-c}
\includegraphics[width=2in]{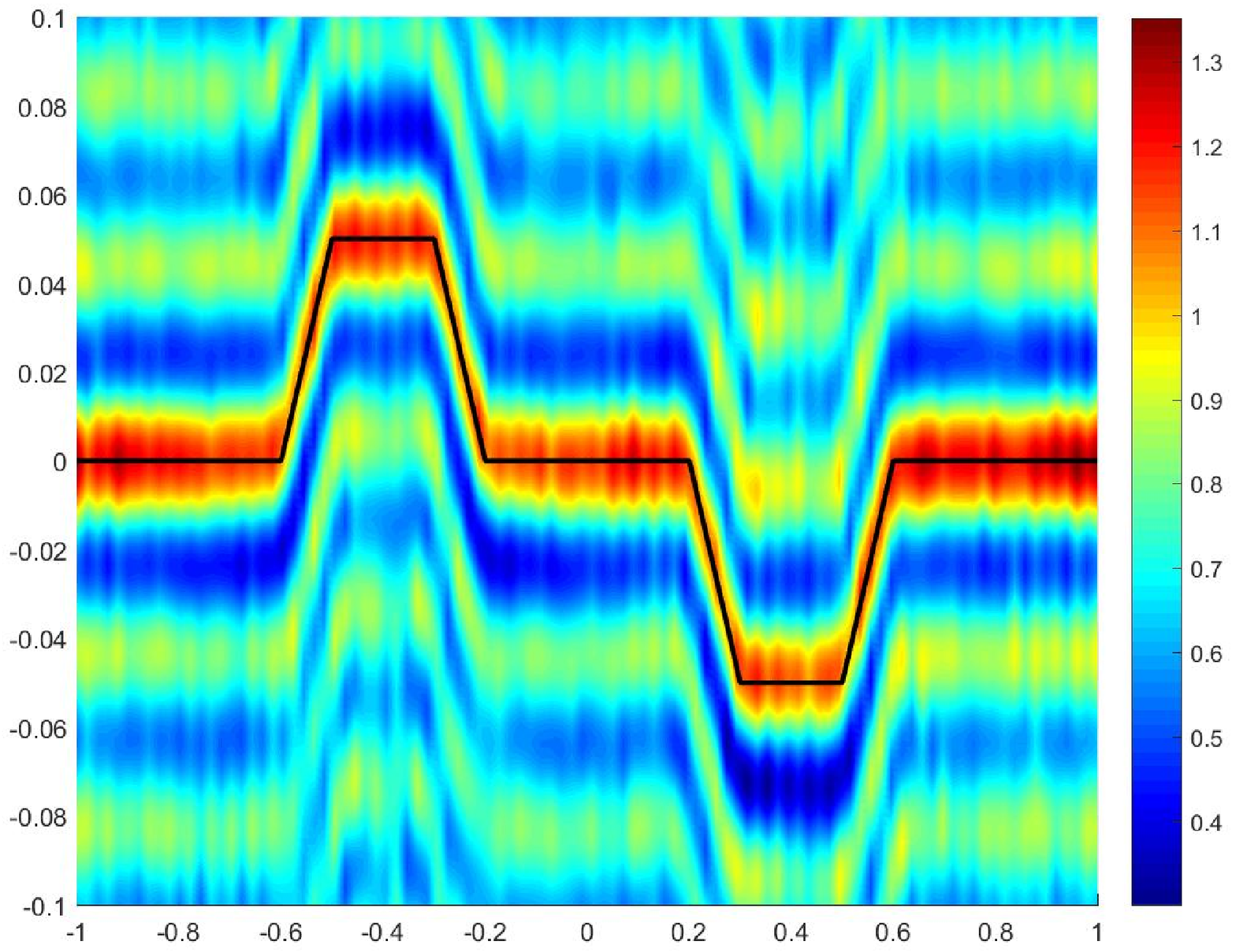}}
\subfigure[\textbf{20\% noise, k=80}]{\label{fig3-d}
\includegraphics[width=2in]{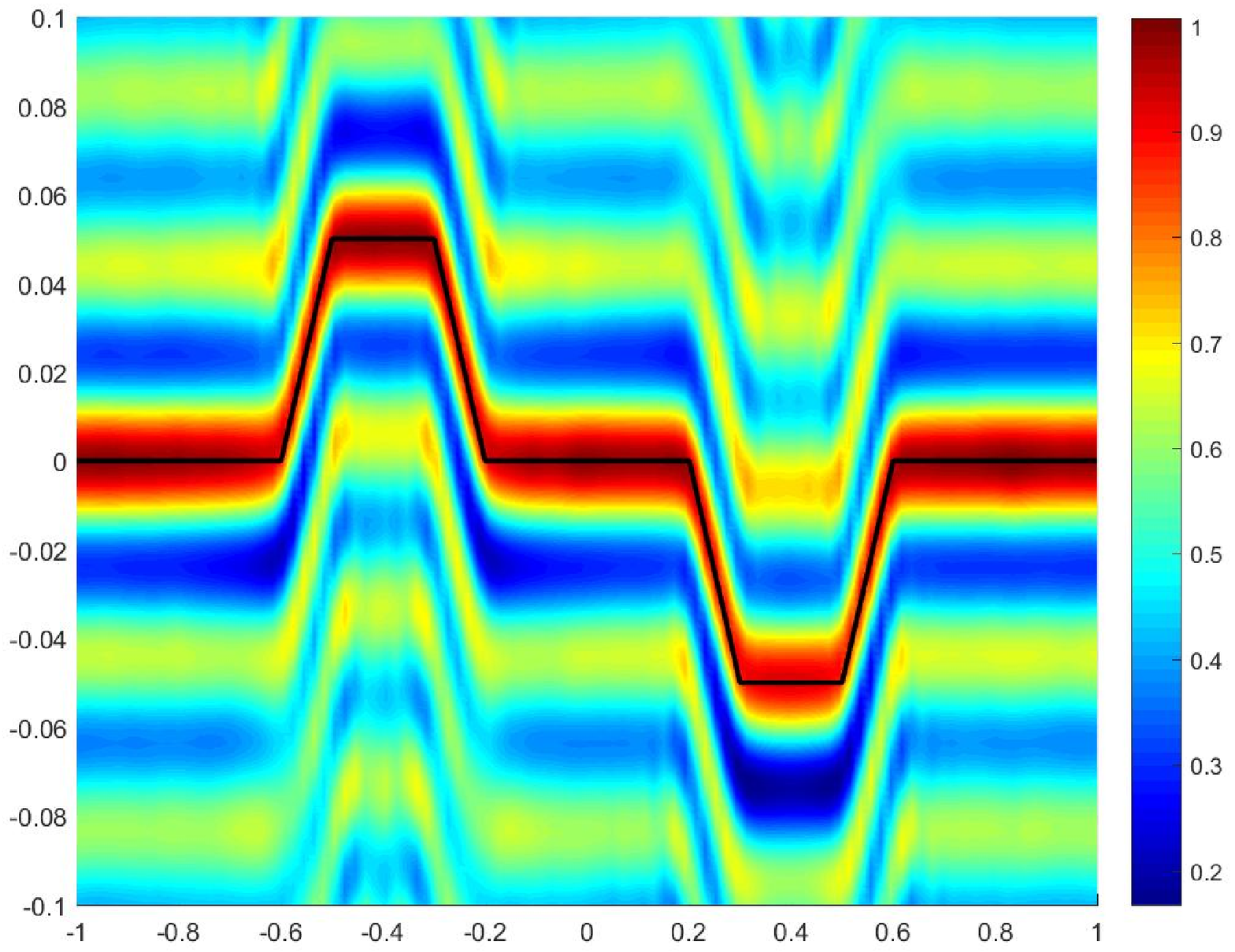}}
\caption{(a)-(c) show the imaging results of $I^{Phaseless}(z)$ with the measured phaseless near-field data,
where the number of the measurement points and the incident directions is chosen to be the same with $M=N=200$.
(d) shows the imaging result of $I^{Full}(z)$ with the measured full far-field data, where
the number of the measurement points and the incident directions is chosen to be the same with $L=N=100$.
The solid line '-' represents the actual curve.
}\label{fig3}
\end{figure}

\textbf{Example 3.} We now consider the case of a multi-scale surface profile given by
\ben
h(x_1)=\left\{\begin{array}{ll}
\ds 0.3\exp\left[16/(25x_1^2-16)\right]\left[0.5+0.1\sin(16\pi x_1)\right]\sin(\pi x_1), &|x_1|<4/5,\\
\ds 0, &|x_1|\geq4/5.
\end{array}\right.
\enn
This profile consists of a macro-scale represented by $0.15\exp[16/(25x_1^2-16)]\sin(\pi x_1)$
and a micro-scale represented by $0.03\exp[16/(25x_1^2-16)]\sin(16\pi x_1)\sin(\pi x_1)$.
We will investigate the effect of the wave number $k$ on the imaging results.
The noise ratio is chosen to be $\delta=20\%$.
We first consider the inverse problem with phaseless near-field data.
The radius of the measurement circle $\pa B^+_R$ is chosen to be $R=4$
and the number of the measurement points and the incident directions is set to be $M=N=400$.
Figure \ref{fig4} presents the imaging results of $I^{Phaseless}(z)$ with the measured
phaseless near-field data with the wave number $k=40,80,120$, respectively.
Second, we consider the inverse problem with full far-field data.
We choose the number of the measurement directions and the incident directions to be $L=N=100$.
Figure \ref{fig5} shows the imaging results of $I^{Full}(z)$ with the measurement
full far-field data with the wave number $k=40,80,120$, respectively.

\begin{figure}[htbp]
  \centering
  \subfigure[\textbf{20\% noise, k=40, R=4}]{
    \includegraphics[width=1.6in]{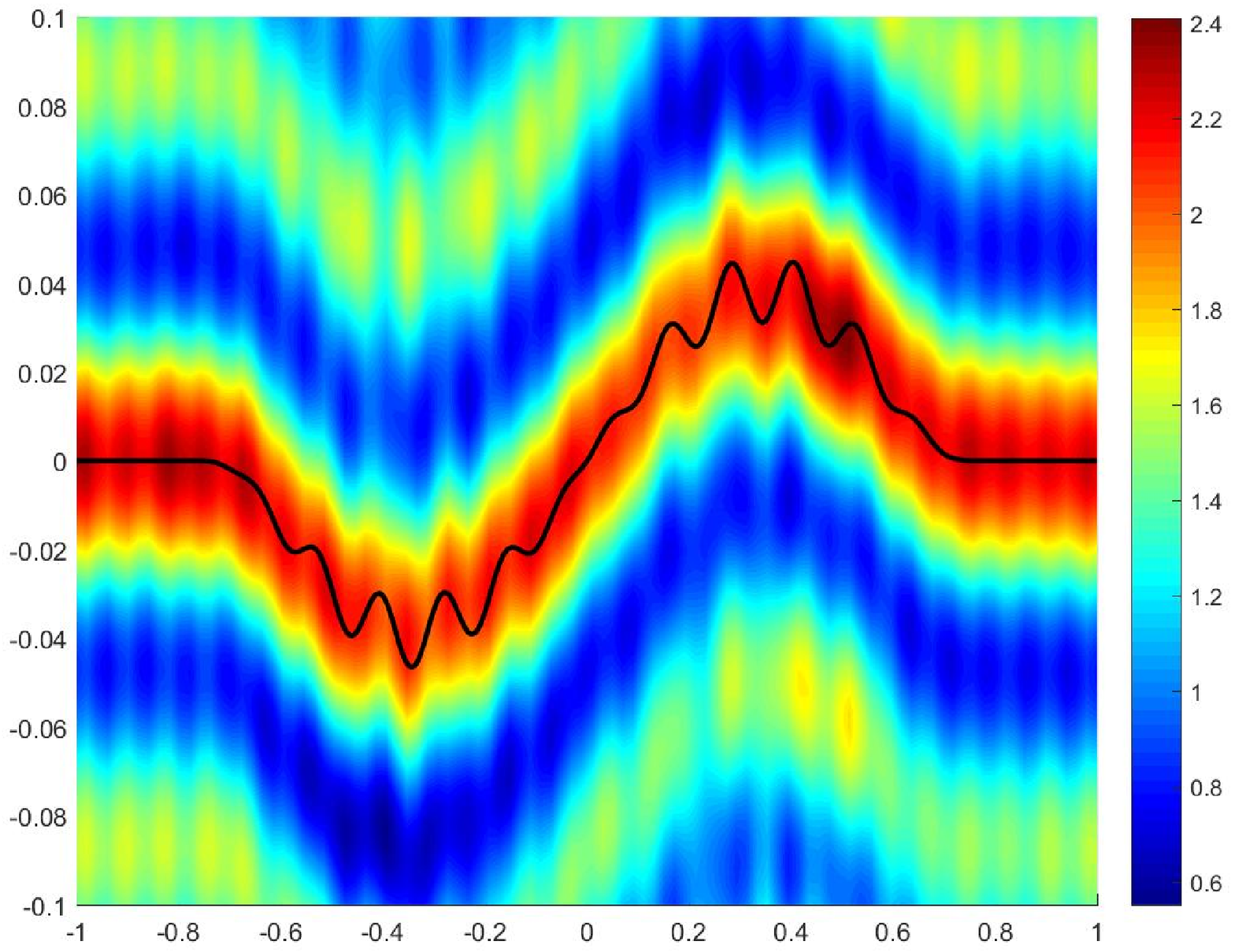}}
  \subfigure[\textbf{20\% noise, k=80, R=4}]{
    \includegraphics[width=1.6in]{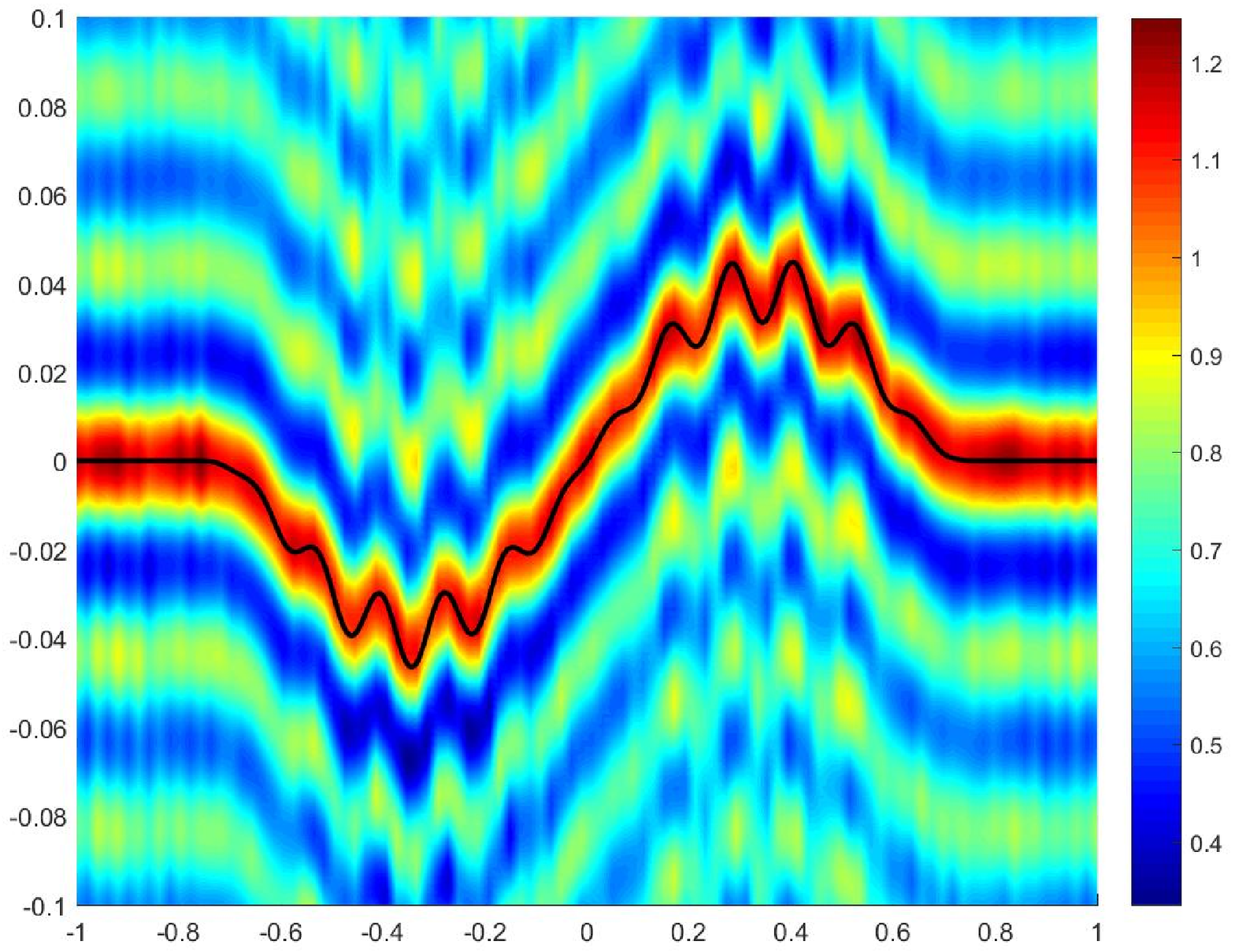}}
  \subfigure[\textbf{20\% noise, k=120, R=4}]{
    \includegraphics[width=1.6in]{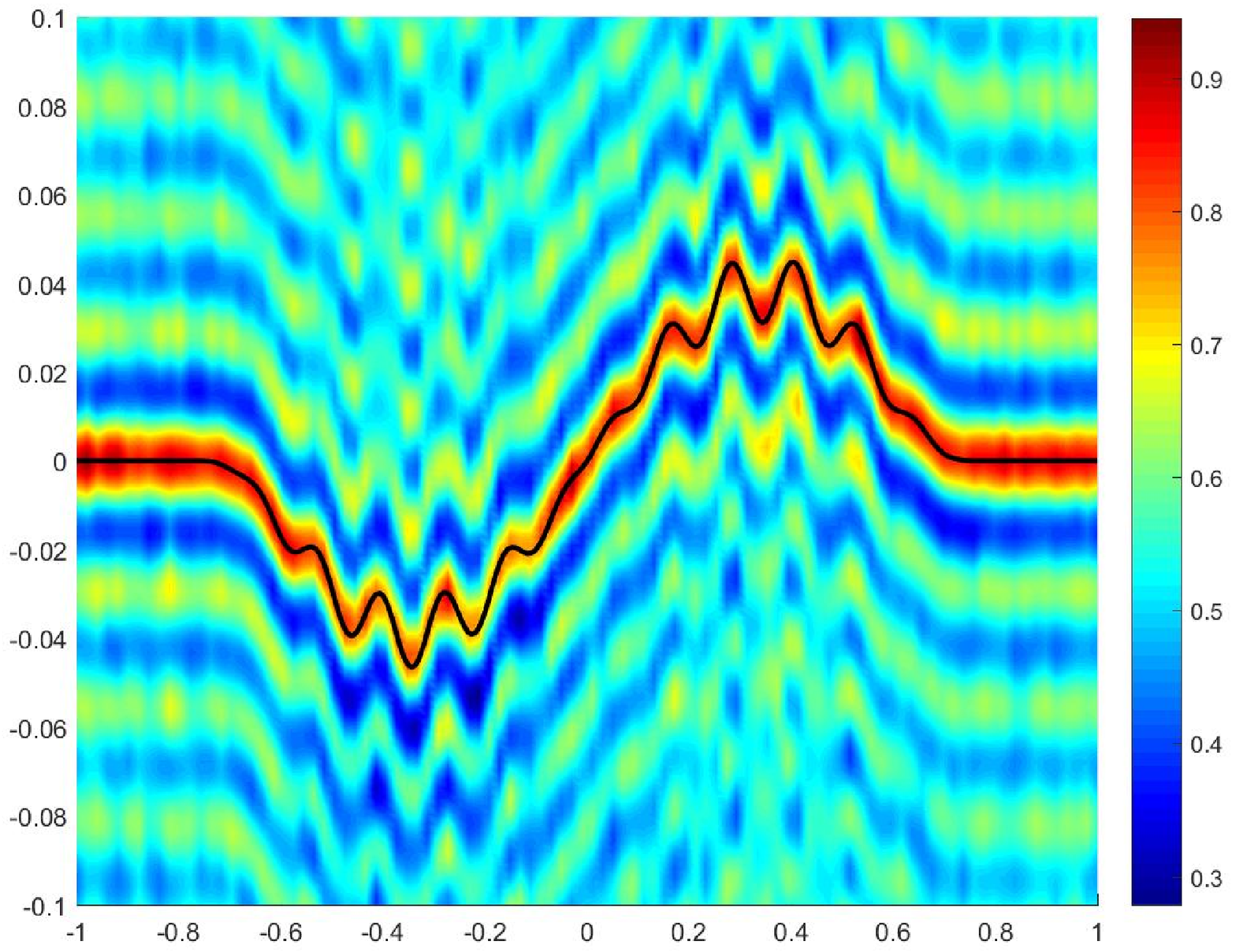}}
\caption{The imaging results of $I^{Phaseless}(z)$ with the measurement phaseless near-field data.
The number of the measurement points and the incident directions is chosen to be the same with $M=N=400$.
The solid line '-' represents the actual curve.
}\label{fig4}
\end{figure}

\begin{figure}[htbp]
  \centering
  \subfigure[\textbf{20\% noise, k=40}]{
    \includegraphics[width=1.6in]{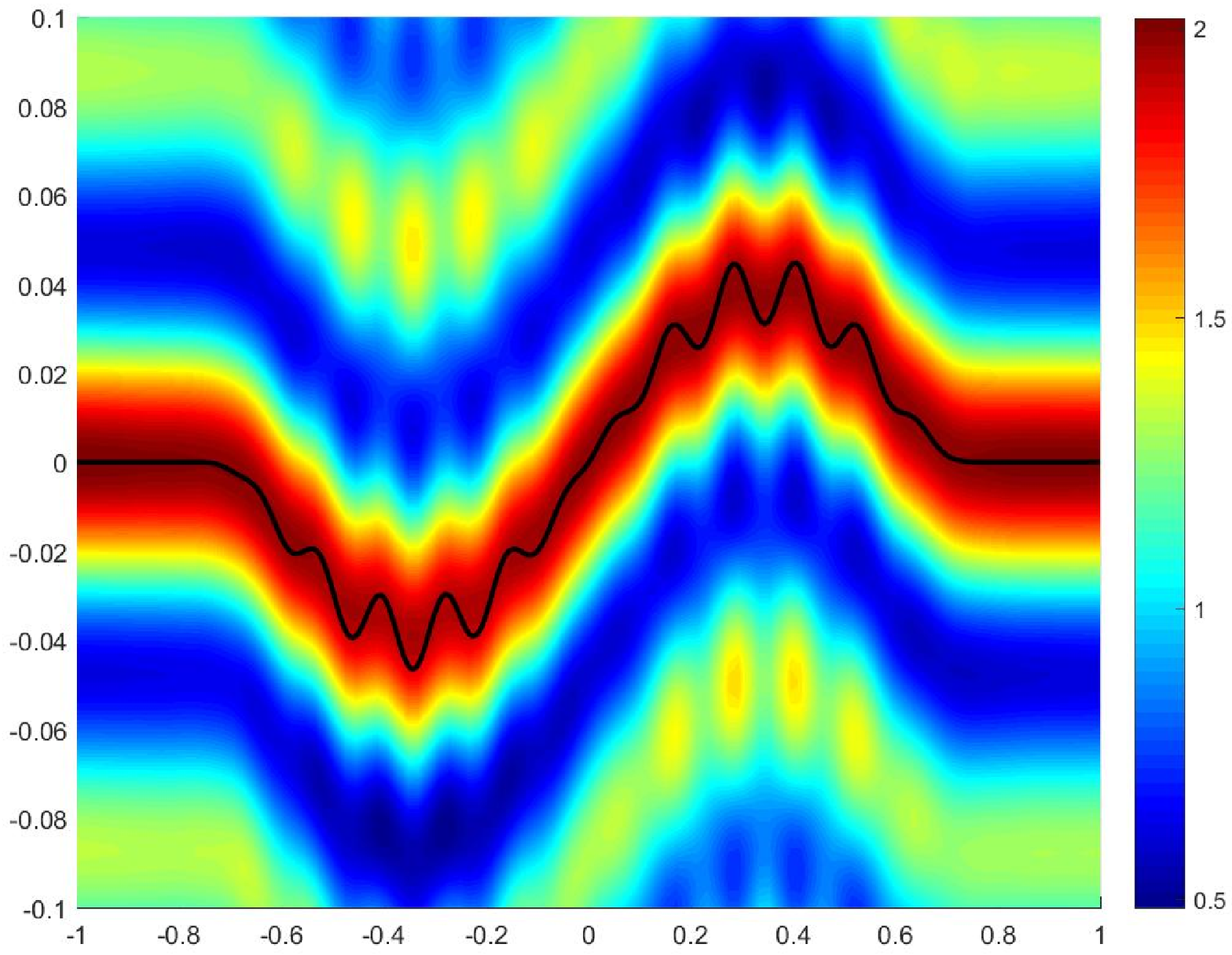}}
  \subfigure[\textbf{20\% noise, k=80}]{
    \includegraphics[width=1.6in]{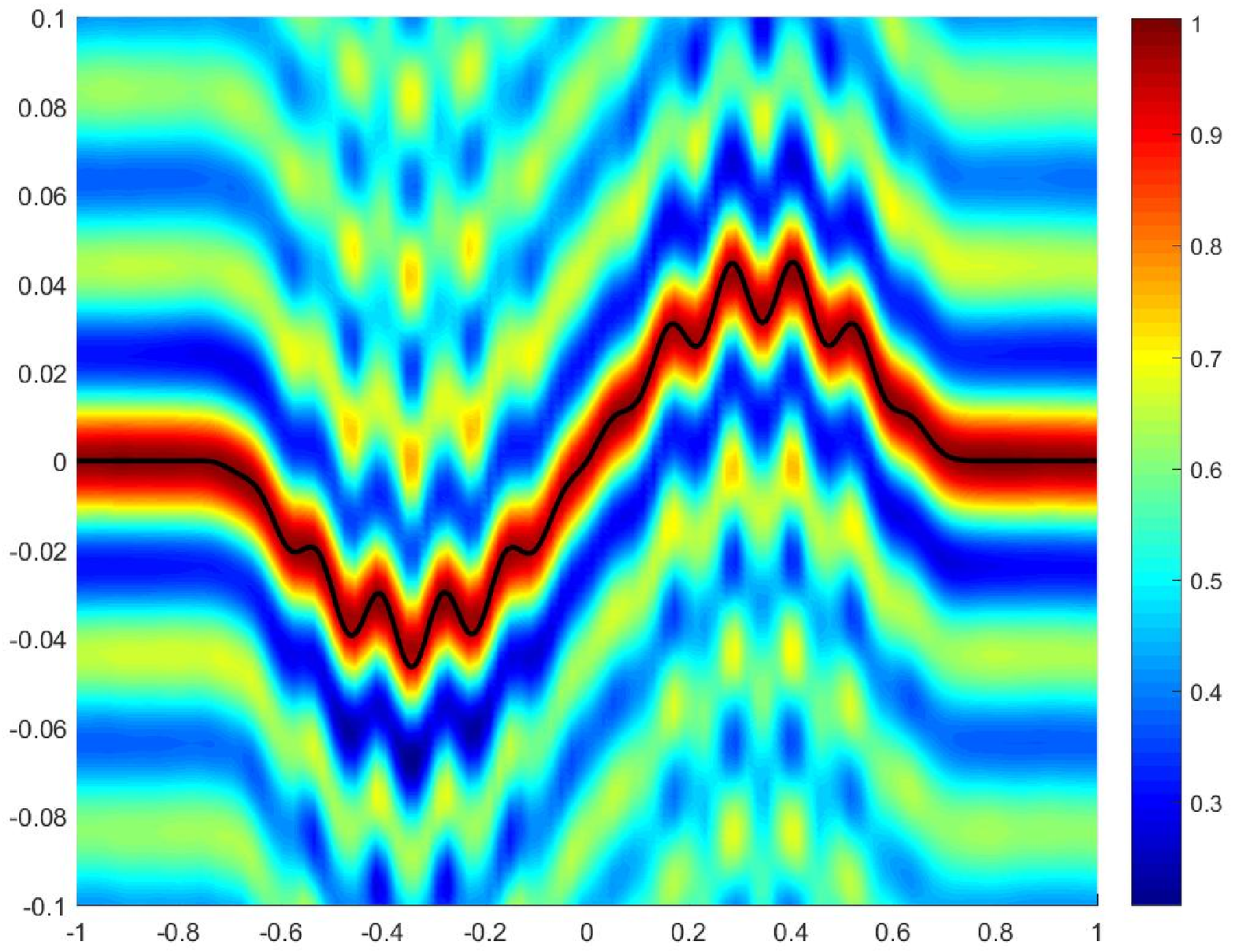}}
  \subfigure[\textbf{20\% noise, k=120}]{
    \includegraphics[width=1.6in]{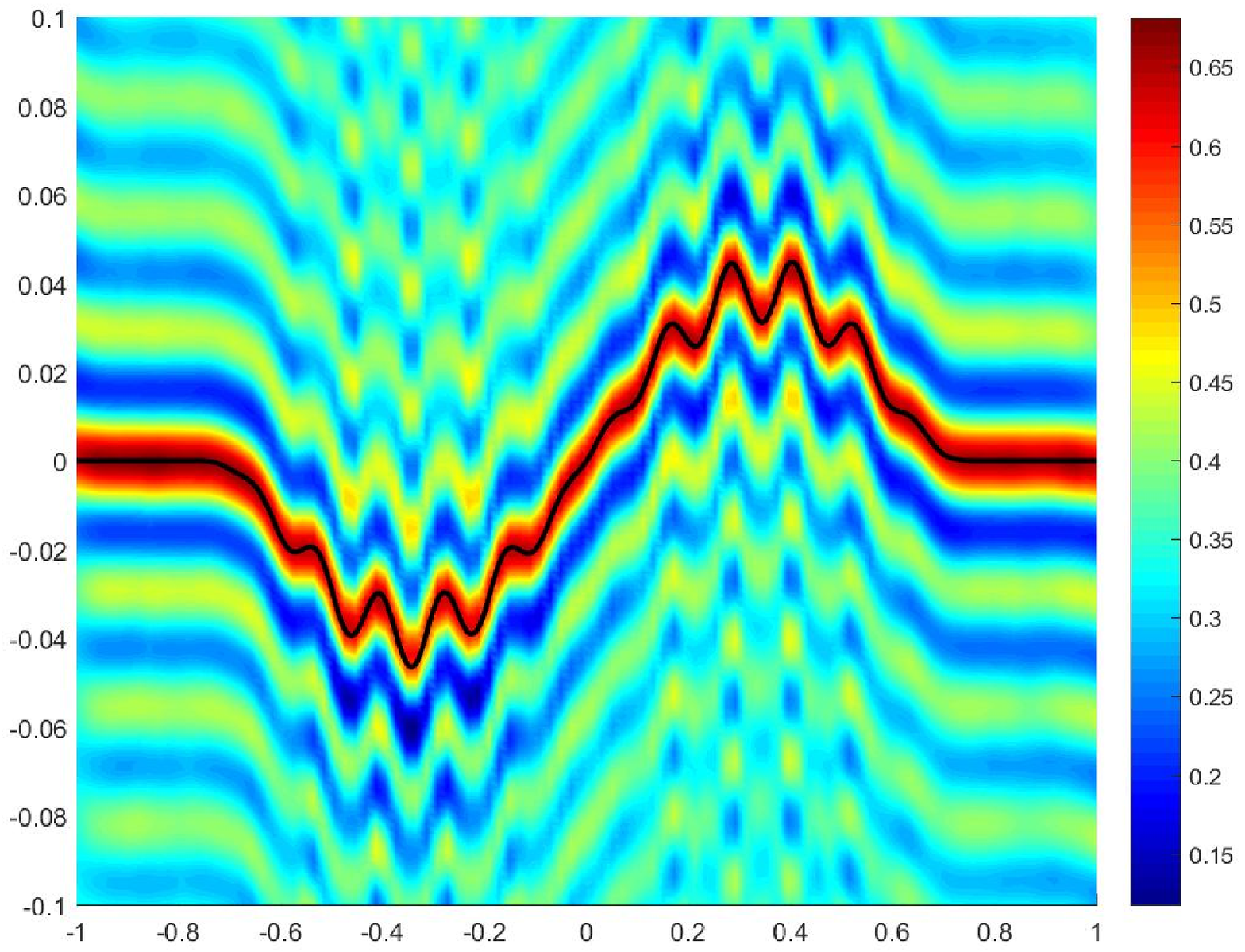}}
\caption{The imaging results of $I^{Full}(z)$ with the measurement full far-field data, where
the number of the measurement points and the incident directions is the same with $L=N=100$.
The solid line '-' represents the actual curve.
}\label{fig5}
\end{figure}

\section{Conclusion}\label{conclusion}

In this paper, we considered the inverse scattering problem by locally rough surfaces with
phaseless near-field data. We have proved that the locally rough surface is uniquely determined
by the phaseless near-field data, generated by a countably infinite number of incident plane waves
and measured on an open domain above the locally rough surface.
A direct imaging method has also been proposed to reconstruct the locally rough surface from
phaseless near-field data generated by incident plane waves and measured on the upper part of
a sufficiently large circle. The theoretical analysis of the imaging method has been given based
on the method of stationary phase and the property of the scattering solution.
As a by-product of the theoretical analysis, a similar direct imaging method with full far-field data
has also been given to reconstruct the locally rough surface and to compare with the imaging method
with phaseless near-field data. As an ongoing project, we are currently trying to extend the results
to the case of incident point sources. In the near future, we hope to consider the more challenging
case of electromagnetic waves.

\section*{Acknowledgements}
\addcontentsline{toc}{section}{Acknowledgements}

This work is partly supported by the NNSF of China grants 91630309, 11501558, 11571355 and 11871466.

\appendix
\renewcommand{\theequation}{\Alph{section}.\arabic{equation}}

\section{The method of stationary phase and proof of Lemma \ref{lem5}}
%\addcontentsline{toc}{section}{Appendix: The method of stationary phase and its application
%to the proof of Lemma \ref{lem5}}
\addcontentsline{toc}{section}{Appendix: The method of stationary phase and proof of Lemma \ref{lem5}}

\setcounter{section}{1}
\setcounter{equation}{0}

In \cite{F74}, the author developed an error theory for the method of stationary phase for integrals
of the form
\be\label{eq43}
I(\gamma)=\int^{b}_{a} e^{i\gamma p(t)}q(t)dt,
\en
where  $a,b\in\R$, $\g$ is a large real parameter, the function $p(t)$ is real and $q(t)$ is either real
or complex. In what follows, we will briefly present some useful results in \cite{F74} and use these
results to prove Lemma \ref{lem5}. For a comprehensive discussion of the method of stationary phase,
the reader is referred to \cite{E55,E56,F74,W01}.

\subsection{Error theory for the method of stationary phase}\label{se5}

We first present an error theory for the method of stationary phase given in \cite{F74}.
Let $a,b\in\R$, let $p(t)$ be a real function and let $q(t)$ be either a real or a complex function.
Assume that $a,b,p(t),q(t)$ are independent of the positive parameter $\g$.
They have the following properties.

(i) In $(a,b)$, $p^{(m+1)}$ and $q^{(m)}(t)$ are continuous, $m$ being a nonnegative integer,
and $p'(t)>0$.

(ii) As $t\rightarrow a$ from the right,
\be\label{eq50}
p(t)\thicksim p(a)+\sum^\infty\limits_{s=0}p_s(t-a)^{s+\mu},\quad
q(t)\thicksim\sum^\infty\limits_{s=0}q_s(t-a)^{s+\la-1},
\en
where the coefficients $p_0$ and $q_0$ are nonzero, and $\mu$ and $\la$ are constants satisfying that
\ben
\mu>0,\;(m+1)\mu+1>\Rt(\la)>0.
\enn
Moreover, the first of these expansions is differentiable $m+1$ times and the second $m$ times.

(iii) $p(b)\equiv\lim_{t\rightarrow b-}\{p(t)\}$ is finite, and each of the functions
\be\label{eq42}
P_s(t)\equiv\left\{\frac{1}{p'(t)}\frac{d}{dt}\right\}^s\frac{q(t)}{p'(t)},
\quad s=0,1,\cdots,m
\en
tends to a finite limit as $t\rightarrow b-$. In particular, (\ref{eq42}) is satisfied if $p^{(m+1)}(t)$
and $q^{(m)}(t)$ are continuous at $b$ and $p'(b)\neq 0$.

In consequence of condition (i) there is a one-to-one relationship between $t$ and the variable $v$,
defined by
\be\label{eq51}
v=p(t)-p(a).
\en
In terms of this variable the integral (\ref{eq43}) transforms into
\ben
\int^b_a e^{i\gamma p(t)}q(t)dt=e^{i\g p(a)}\int^{p(b)-p(a)}_0 e^{i\g v}f(v)dv,
\enn
in which
\be\label{eq52}
f(v)=q(t)/p'(t)=P_0(t).
\en
Again, condition (i) shows that $f(v)$ and its first $m$ derivatives are continuous when $0<v<p(b)-p(a)$.
For small $v$, $f(v)$ can be expanded in asymptotic series of the form
\be\label{eq53}
f(v)\thicksim \sum^\infty\limits_{s=0}a_s v^{(s+\lambda-\mu)/\mu}.
\en
The coefficients $a_s$ depend on $p_s$ and $q_s$, and may be found by standard procedures of
reverting series. In particular,
\be\label{eq54}
a_0=\frac{q_0}{\mu p^{\lambda/\mu}_0},\quad
a_1=\left\{\frac{q_1}{\mu}-\frac{(\lambda+1)p_1q_0}{\mu^2p_0}\right\}
\frac{1}{p^{(\lambda+1)/\mu}_0}.
\en

The following theorem gives an asymptotic expansion of the integral (\ref{eq43}) with an error bound
(see Theorem 1 and estimates (6.3) and (6.7) in \cite{F74}).

\begin{theorem}[Theorem 1 and estimates (6.3) and (6.7) in \cite{F74}]\label{thm-a}
Assume the conditions and notation of this section, and let $n$ be a nonnegative integer satisfying
\be\label{eq55}
m\mu-\lambda\leq n<(m+1)\mu-\lambda+1\quad (\lambda\;{\rm real}),
\en
or
\be\label{eq55+}
m\mu-\Rt(\lambda)< n<(m+1)\mu-\Rt(\la)+1\quad (\lambda\;{\rm complex}).
\en
If $p(b)<\infty$ then we have
\be\label{eq56}
\int^b_a e^{i\g p(t)}q(t)dt=e^{i\g p(a)}\sum^{n-\nu}\limits_{s=0}
\exp\left\{\frac{(s+\la)\pi i}{2\mu}\right\}\G\left(\frac{s+\la}{\mu}\right)\frac{a_s}{\g^{(s+\la)/\mu}}\no\\
-e^{i\g p(b)}\sum^{m-1}\limits_{s=0}P_s(b)\left(\frac{i}{\g}\right)^{s+1}+\delta_{m,n}(\g)-\vep_{m,n}(\g).
\en
Here, $\nu=0$ when $n=m\mu-\la$, and $\nu=1$ in all other cases. As usual, empty sums are understood to
be zero. Further, the error terms $\delta_{m,n}$ and $\vep_{m,n}$ satisfy
\be\label{eq57}
|\delta_{m,n}(\g)|\le\left[|Q_{m+1,n}(a)|+|Q_{m+1,n}(b)|+\mathscr{V}_{a,b}\{Q_{m+1,n}(t)\}\right]\g^{-m-1}
\en
provided that the right-hand side is finite, and
\be\label{eq62}
|\vep_{m,n}(\g)|\leq\frac{2}{\g^{m+1}}\sum^{n-1}\limits_{s=0}\frac{\G\{(s+\la)/\mu\}}{|\G\{(s+\la-m\mu)/\mu\}|}
\frac{|a_s|}{\{p(b)-p(a)\}^{(m\mu+\mu-s-\la)/\mu}}.\;\;
\en
Here, $Q_{m+1,n}$ is given by
\be\label{eq58}
&&Q_{m+1,n}(t)\no\\
&&=P_{m}(t)-\sum^{n-1}\limits_{s=0}\frac{\G\{(s+\la)/\mu\}}{\G\{[s+\la+\mu-(m+1)\mu]/\mu\}}
\frac{a_s}{\{p(t)-p(a)\}^{[(m+1)\mu-s-\la]/\mu}},\quad
\en
and $\mathscr{V}_{a,b}\{Q_{m+1,n}(t)\}$ denotes the total variation of
the function $Q_{m+1,n}(t)$ which is given by
\ben
\mathscr{V}_{a,b}\{Q_{m+1,n}(t)\}=\int^b_a|Q'_{m+1,n}(t)|dt.
\enn
\end{theorem}

\subsection{Proof of Lemma \ref{lem5}}

In this section, we prove Lemma \ref{lem5}, employing Theorem \ref{thm-a}.
To do this, we need to estimate the function $U_i,$ $i=1,2,3,$ defined in (\ref{eq44})-(\ref{eq46}).

\begin{lemma}\label{lem7}
%Let $x\in\ov{D_+}$ and $z\in\R^2$, we have
Let $x\in{D_+}$ and $z\in\R^2$. Then we have
\be\label{eq93}
U_1(x,z)=\frac{e^{ik|x|}}{|x|^{1/2}}\int_{\Sp^1_-}u^\infty(\hat{x},d)e^{-ikz\cdot d}ds(d)+U_{1,Res}(x,z)
\en
with
\be\label{eq95}
|U_{1,Res}(x,z)|\le\frac{C}{|x|^{3/2}}\quad\mbox{as}\;\;|x|\rightarrow+\infty,
\en
where $C>0$ is a constant independent of $x$ and $z$.
\end{lemma}

\begin{proof}
Multiply (\ref{eq92}) by $e^{-ik z\cdot d}$ and integrate with respect to $d$ over $\Sp^1_-$ to obtain
(\ref{eq93}) with $U_{1,Res}$ being given by
\ben
U_{1,Res}(x,z):=\int_{\Sp^1_-}u^s_{Res}(x,d)e^{-ik z\cdot d}ds(d).
\enn
The inequality (\ref{eq95}) then follows from (\ref{eq94}). The lemma is thus proved.
\end{proof}

\begin{lemma}\label{lem6}
Let $x\in \R^2_+$ and $z\in\R^2$. Write $\hat{x}=x/|x|=(\cos\theta_{\hat{x}},\sin\theta_{\hat{x}})$
with $\theta_{\hat{x}}\in(0,\pi)$. Then we have that for $\theta_{\hat{x}}\in(0,\pi)$,
\be\label{eq47}
U_2(x,z)&=&-\frac{e^{ik|x|}}{|x|^{1/2}}e^{-\frac{\pi}{4}i}\left(\frac{2\pi}{k}\right)^{1/2}e^{-ik\hat{x}\cdot z'}
+U_{2,Res}(x,z),\\ \label{eq49}
U_3(x,z)&=&-\frac{e^{ik|x|}}{|x|^{1/2}}e^{-\frac{\pi}{4}i}\left(\frac{2\pi}{k}\right)^{1/2}e^{-ik\hat{x}\cdot z}
+U_{3,Res}(x,z),
\en
where
\be\label{eq48}
|U_{j,Res}(x,z)|&\le& C\Bigg(|z|+\frac{1}{|\sin\frac{\theta_{\hat{x}}}{2}|}
+\frac{1}{|\sin\frac{\pi-\theta_{\hat{x}}}{2}|}+\frac{1}{|\sin{\theta_{\hat{x}}}|}
+\int^{\theta_{\hat{x}}}_0\frac{(1+|z|)^3 t^2}{\sin^2 t}dt\no\\
&&\qquad\;\; +\int^{\pi-\theta_{\hat{x}}}_0\frac{(1+|z|)^3 t^2}{\sin^2 t}dt\Bigg)\frac{1}{|x|},\;\;j=2,3
\en
for large $|x|$. Here, $C>0$ is a constant independent of $x$ and $z$.
\end{lemma}

\begin{proof}
We only consider the case for $U_2(x,z)$. The case for $U_3(x,z)$ can be proved similarly.

For $d\in\Sp^1_-$ and $z\in\R^2$, let $\theta_d,\theta_{\hat{z}}$ be the real numbers as defined at
the end of Section \ref{sec1}. Then we have
\ben
U_2(x,z)=-\int_{\Sp^1_-}e^{ik(x\cdot d'-z\cdot d)}ds(d)
=-\int^{2\pi}_\pi e^{ik|x|\cos(\theta_d+\theta_{\hat{x}})}e^{-ik|z|\cos(\theta_d-\theta_{\hat{z}})}d\theta_d.
\enn
A straightforward calculation gives
\be\label{eq66}
-\ov{U_2(x,z)}&=&\int^{2\pi-\theta_{\hat{x}}}_\pi e^{-ik|x|\cos(\theta_d+\theta_{\hat{x}})}
   e^{ik|z|\cos(\theta_d-\theta_{\hat{z}})}d\theta_d\no\\
&&+\int^{2\pi}_{2\pi-\theta_{\hat{x}}}e^{-ik|x|\cos(\theta_d+\theta_{\hat{x}})}
   e^{ik|z|\cos(\theta_d-\theta_{\hat{z}})}d\theta_d\no\\
&=&\int^{\pi-\theta_{\hat{x}}}_0 e^{-ik|x|\cos t}e^{ik|z|\cos(t+\theta_{\hat{x}}+\theta_{\hat{z}})}dt
   +\int^{\theta_{\hat{x}}}_0 e^{-ik|x|\cos t}e^{ik|z|\cos(t-\theta_{\hat{x}}-\theta_{\hat{z}})}dt\no\\
&:=&I_1(x,z)+I_2(x,z).
\en

We first estimate $I_1(x,z)$. Let $\g=|x|,a=0,b=\pi-\theta_{\hat{x}},$
$p(t)=-k\cos t$ and $q(t)=e^{ik|z|\cos(t+\theta_{\hat{x}}+\theta_{\hat{z}})}$.
Then it is easy to verify that $a,b,p(t),q(t)$ satisfy the assumptions in Section \ref{se5}.
In particular, $p(t),q(t)$ satisfy Assumption (\ref{eq50}) with $\mu=2,\la=1,p_0=k/2$ and
$q_0=e^{ik|z|\cos(\theta_{\hat{x}}+\theta_{\hat{z}})}$, and the function $P_0$ defined in (\ref{eq42})
is given by
\be\label{eq59}
P_0(t)=q(t)/p'(t)=e^{ik|z|\cos(t+\theta_{\hat{x}}+\theta_{\hat{z}})}/(k\sin t).
\en
Let the relationship between $t$ and $v$ be given by (\ref{eq51}) and let $f(v)$ be the function
defined in (\ref{eq52}). Then $f(v)$ has the form (\ref{eq53}). In particular,
it follows from (\ref{eq54}) that the coefficient
$a_0=q_0/(\mu p^{\la/\mu}_0)=e^{ik|z|\cos(\theta_{\hat{x}}+\theta_{\hat{z}})}/(2k)^{1/2}$.

Choose $m=0,n=1$. Then $m,n$ satisfy the condition (\ref{eq55}). Thus it follows from (\ref{eq56}) that
\be\label{eq63}
I_1(x,z)&=&\int^b_a e^{i|x|p(t)}q(t)dt=e^{i|x|p(0)}exp\left\{\frac{\la\pi i}{2\mu}\right\}\G
\left(\frac{\la}{\mu}\right)\frac{a_0}{|x|^{\la/\mu}}+\delta_{0,1}(|x|)-\vep_{0,1}(|x|)\no\\
&=&\frac{e^{-ik|x|}}{|x|^{1/2}}e^{\pi i/4}\left(\frac{\pi}{2k}\right)^{1/2}
e^{ik|z|\cos(\theta_{\hat{x}}+\theta_{\hat{z}})}+\delta_{0,1}(|x|)-\vep_{0,1}(|x|)\no\\
&=&\frac{e^{-ik|x|}}{|x|^{1/2}}e^{\pi i/4}\left(\frac{\pi}{2k}\right)^{1/2}
e^{ik\hat{x}\cdot z'}+\delta_{0,1}(|x|)-\vep_{0,1}(|x|).
\en
Further, by (\ref{eq58}) and (\ref{eq59}) we have
\ben
Q_{1,1}(t)&=&P_0(t)-\frac{a_0}{(p(t)-p(a))^{1/2}}\\
&=&\frac{1}{k\sin t}\left( e^{ik|z|\cos(t+\theta_{\hat{x}}+\theta_{\hat{z}})}
-e^{ik|z|\cos(\theta_{\hat{x}}+\theta_{\hat{z}})}\cos\frac{t}{2}\right)
\enn
for $t\in(0,\pi-\theta_{\hat{x}})$. It is easy to see that
\be\label{eq60}
|Q_{1,1}(0)|\leq|z|,\quad
|Q_{1,1}(\pi-\theta_{\hat{x}})|\leq\frac{2}{k|\sin(\pi-\theta_{\hat{x}})|}
=\frac{2}{k|\sin\theta_{\hat{x}}|}
\en
and
\ben
Q'_{1,1}(t)&=&\frac{\left(\frac{1}{2}\sin\frac{t}{2}\sin t+\cos\frac{t}{2}\cos t\right)
e^{ik|z|\cos(\theta_{\hat{x}}+\theta_{\hat{z}})}}{k\sin^2 t}\no\\
&&\quad-\frac{\left(ik|z|\sin(t+\theta_{\hat{x}}+\theta_{\hat{z}})\sin t+\cos t\right)
e^{ik|z|\cos(t+\theta_{\hat{x}}+\theta_{\hat{z}})}}{k\sin^2 t}
:=\frac{h(t)}{k\sin^2 t}.
\enn
By a straightforward calculation, it is derived that
\ben
h(0)=h'(0)=0\;\;\; |h''(t)|\leq C(1+|z|)^3,\;\;\;t\in\R.
\enn
Then, by the Taylor expansion we obtain that for $t\in(0,\pi-\theta_{\hat{x}})$,
\be\label{eq61}
|Q'_{1,1}(t)|\leq C\frac{(1+|z|)^3 t^2}{\sin^2 t}.
\en
Combining (\ref{eq57}), (\ref{eq60}) and (\ref{eq61}) gives
\be\label{eq64}
|\delta_{0,1}(|x|)|&\le&\left[|Q_{1,1}(0)|+|Q_{1,1}(\pi-\theta_{\hat{x}})|
+\mathscr{V}_{0,\pi-\theta_{\hat{x}}}\{Q_{1,1}(t)\}\right]|x|^{-1}\no\\
&\le& C\left(|z|+\frac{1}{|\sin{\theta_{\hat{x}}}|}
+\int^{\pi-\theta_{\hat{x}}}_0\frac{(1+|z|)^3 t^2}{\sin^2 t}dt\right)\frac{1}{|x|}.
\en
Further, it follows from (\ref{eq62}) that
\be\label{eq65}
|\vep_{0,1}|\leq\frac{2}{|x|}\frac{|a_0|}{(p(b)-p(a))^{1/2}}=
\frac{1}{k|x|}\frac{1}{\left|\sin\frac{\pi-\theta_{\hat{x}}}{2}\right|}.
\en
Let $I_{1,Res}(x,z):=\delta_{0,1}(|x|)-\vep_{0,1}(|x|)$. Then combining (\ref{eq63}), (\ref{eq64})
and (\ref{eq65}) yields
\be\label{eq67}
I_1(x,z)=\frac{e^{-ik|x|}}{|x|^{1/2}}e^{\pi i/4}\left(\frac{\pi}{2k}\right)^{1/2}e^{ik\hat{x}\cdot z'}
+I_{1,Res}(x,z)
\en
with
\be\label{eq68}
|I_{1,Res}(x,z)|\le C\left(|z|+\frac{1}{|\sin\frac{\pi-\theta_{\hat{x}}}{2}|}
+\frac{1}{|\sin\theta_{\hat{x}}|}
+\int^{\pi-\theta_{\hat{x}}}_0\frac{(1+|z|)^3t^2}{\sin^2 t}dt\right)\frac{1}{|x|}.\qquad\;
\en

Similarly, for $I_2(x,z)$ we have
\be\label{eq69}
I_2(x,z)=\frac{e^{-ik|x|}}{|x|^{1/2}}e^{\pi i/4}
\left(\frac{\pi}{2k}\right)^{1/2}e^{ik\hat{x}\cdot z'}
+I_{2,Res}(x,z),
\en
where
\be\label{eq70}
|I_{2,Res}(x,z)|\leq C\left(|z|+\frac{1}{|\sin\frac{\theta_{\hat{x}}}{2}|}+\frac{1}{|\sin\theta_{\hat{x}}|}
+\int^{\theta_{\hat{x}}}_0\frac{(1+|z|)^3t^2}{\sin^2 t}dt\right)\frac{1}{|x|}.
\en

Now, let $U_{2,Res}(x,z):=-(\ov{I_{1,Res}(x,z)}+\ov{I_{2,Res}(x,z)})$.
Then the equation (\ref{eq47}) with the estimate (\ref{eq48}) follows from (\ref{eq66}), (\ref{eq67}),
(\ref{eq68}), (\ref{eq69}) and (\ref{eq70}). The proof is thus complete.
\end{proof}

\begin{lemma}\label{lem9}
Let $z\in\R^2$ and let  $U_{j,Res}(x,z),j=2,3,$ be the functions defined in Lemma \ref{lem6}.
Assume that $\delta>0$ is small enough and $R>0$ is large enough. Then, for any
$x=R(\cos\theta_{\hat{x}},\sin\theta_{\hat{x}})$ with $\theta_{\hat{x}}\in[\delta,\pi-\delta]$ we have
\be\label{eq71}
|U_{j,Res}(x,z)|\leq C\frac{(1+|z|)^3}{R\delta},\quad j=2,3,
\en
where $C>0$ is a constant independent of $x$ and $z$.
\end{lemma}

\begin{proof}
For $\theta_{\hat{x}}\in[\delta,\pi-\delta]$ with $\delta$ small enough we have
\ben
\frac{1}{\left|\sin\frac{\theta_{\hat{x}}}{2}\right|}+\frac{1}{\left|\sin\frac{\pi-\theta_{\hat{x}}}{2}\right|}
+\frac{1}{\left|\sin{\theta_{\hat{x}}}\right|}\le\frac{C}{\delta}
\enn
and
\ben
&&\int^{\theta_{\hat{x}}}_0\frac{(1+|z|)^3 t^2}{\sin^2 t}dt
+\int^{\pi-\theta_{\hat{x}}}_0\frac{(1+|z|)^3 t^2}{\sin^2 t}dt\\
&&\qquad\le 2\int^{\pi-\delta}_0\frac{(1+|z|)^3 t^2}{\sin^2 t}dt
\le C(1+|z|)^3\left(1+\int^{\pi-\delta}_{\pi/2}\frac{1}{\sin^2 t}dt\right)\\
&&\qquad= C(1+|z|)^3\left(1+\int^{\pi/2}_\delta\frac{1}{\sin^2 t}dt\right)
\le C\frac{(1+|z|)^3}{\delta}.
\enn
This, together with Lemma \ref{lem6}, implies the inequality (\ref{eq71}).
The proof is thus complete.
\end{proof}

We are now ready to prove Lemma \ref{lem5}.

{\em Proof of Lemma $\ref{lem5}$.}
For arbitrarily fixed $z\in\R^2$, let $\delta=R^{-1/4}$ with $R>0$ large enough. Define
$\pa B^+_{R,\delta}:=\{x=R(\cos\theta_{\hat{x}},\sin\theta_{\hat{x}})\;|\;
\theta_{\hat{x}}\in(0,\delta)\cup(\pi-\delta,\pi)\}$ and define
\be\label{eq72}
U_0(x,z)&:=&\frac{e^{ik|x|}}{|x|^{1/2}}\left[\int_{\Sp^1_-}u^\infty(\hat{x},d)e^{-ikz\cdot d}ds(d)\right.\no\\
&&\hspace{2cm}\left.-\left(\frac{2\pi}{k}\right)^{1/2}e^{-\pi i/4}
  (e^{-ik\hat{x}\cdot z'}+e^{-ik\hat{x}\cdot z})\right],\\
U_{Res}(x,z)&:=&\sum^{3}\limits_{j=1}U_{j,Res}(x,z), \label{eq72+}
\en
where $U_{j,Res}(x,z),j=1,2,3,$ are given in (\ref{eq93}), (\ref{eq47}) and (\ref{eq49}), respectively.
Then it follows from Lemmas \ref{lem7} and \ref{lem6} that $U(x,z)=U_0(x,z)+U_{Res}(x,z).$
Now, by the definition of $U_0(x,z)$ and $U(x,z)$, and using Lemmas \ref{le3} and \ref{le4} we get
\be\label{eq77}
|U_0(x,z)|\leq C\frac{1}{R^{1/2}},\quad |U(x,z)|\leq C\frac{1+|z|}{R^{1/2}}\quad\forall x\in\pa B^+_R,
\en
which yields
\be\label{eq78}
|U_{Res}(x,z)|\leq C\frac{1+|z|}{R^{1/2}}\quad\forall x\in\pa B^+_R.
\en
On the other hand, by Lemmas \ref{lem7} and \ref{lem9} and on noting that $\delta=R^{-1/4}$
it follows that
\be\label{eq79}
|U_{Res}(x,z)|\leq C_1\left(\frac{1}{R^{3/2}}+\frac{(1+|z|)^3}{\delta R}\right)
\leq C\frac{(1+|z|)^3}{R^{3/4}}\quad\forall x\in\pa B^+_R\ba\pa B^+_{R,\delta}.\;\quad
\en

We now prove (\ref{eq76}) for the function $F_{0,Res}(x,z)$ defined in Lemma \ref{lem5}.
Since $F(R,z)=F_0(z)+F_{0,Res}(R,z)$, and by the definition of $F(R,z)$, $F_0(z)$ and $U_0(x,z)$
(see (\ref{eq15}), (\ref{eq27}) and (\ref{eq72})), we have
\ben
&&\int_{\pa B^+_R}|U(x,z)|^2dx\\
&&=\int_{\Sp^1_+}\left|\int_{\Sp^1_-}u^\infty(x,d)e^{-ik z\cdot d}ds(d)-\left(\frac{2\pi}{k}\right)^{1/2}
e^{-\pi i/4}\left(e^{-ik\hat{x}\cdot z'}+e^{-ik\hat{x}\cdot z}\right)\right|^2ds(\hat{x})\\
&&\qquad+F_{0,Res}(R,z)=\int_{\pa B^+_R}|U_0(x,z)|^2dx+ F_{0,Res}(R,z).
\enn
Thus we have
\be\label{eq73}
F_{0,Res}(x,z)=\int_{\pa B^+_R}U_0(x,z)\ov{U_{Res}(x,z)}dx+\int_{\pa B^+_R}U_{Res}(x,z)\ov{U(x,z)}dx.
\en
From (\ref{eq77}), (\ref{eq78}) and (\ref{eq79}) it follows that
\be\label{eq74}
&&\left|\int_{\pa B^+_R}U_0(x,z)\ol{U_{Res}(x,z)}dx\right|\no\\
&&\qquad\le\left|\int_{\pa B^+_{R,\delta}}U_0(x,z)\ov{U_{Res}(x,z)}dx\right|
+\left|\int_{\pa B^+_R\ba\pa B^+_{R,\delta}}U_0(x,z)\ov{U_{Res}(x,z)}dx\right|\no\\
&&\qquad\le CR\delta\frac{1}{R^{1/2}}\frac{1+|z|}{R^{1/2}}
+C R\frac{1}{R^{1/2}}\frac{(1+|z|)^3}{\delta R}\no\\
&&\qquad= C\left((1+|z|)\delta+\frac{(1+|z|)^3}{\delta R^{1/2}}\right)
\en
and
\be\label{eq75}
&&\left|\int_{\pa B^+_R}U_{Res}(x,z)\ol{U(x,z)}dx\right|\no\\
&&\qquad\le\left|\int_{\pa B^+_{R,\delta}}U_{Res}(x,y)\ov{U(x,z)}dx\right|
+\left|\int_{\pa B^+_R\ba\pa B^+_{R,\delta}}U_{Res}(x,z)\ov{U(x,z)}dx\right|\no\\
&&\qquad\le CR\delta\frac{1+|z|}{R^{1/2}}\frac{1+|z|}{R^{1/2}}
+CR\frac{(1+|z|)^3}{\delta R}\frac{1+|z|}{R^{1/2}}\no\\
&&\qquad= C\left((1+|z|)^2\delta+\frac{(1+|z|)^4}{\delta R^{1/2}}\right).
\en
Combining (\ref{eq73}), (\ref{eq74}) and (\ref{eq75}) gives
\ben
|F_{0,Res}(R,z)|\leq C\left((1+|z|)^2\delta+\frac{(1+|z|)^4}{\delta R^{1/2}}\right).
\enn
This, combined with the fact that $\delta=R^{-1/4}$, yields (\ref{eq76}).
Lemma \ref{lem5} is thus proved.
$\Box$

\end{document}